\documentclass[reqno]{amsart}
\usepackage{amssymb}
\usepackage{graphicx}

\usepackage[usenames, dvipsnames]{color}
\usepackage{verbatim}
\usepackage{mathrsfs}
\usepackage{bm}
\usepackage{cite}

\numberwithin{equation}{section}

\newtheorem{theorem}{Theorem}[section]
\newtheorem{corollary}[theorem]{Corollary}
\newtheorem{lemma}[theorem]{Lemma}
\newtheorem{prop}[theorem]{Proposition}

\theoremstyle{definition}
\newtheorem{remark}[theorem]{Remark}

\theoremstyle{definition}

\theoremstyle{definition}

\makeatletter
\def\dashint{\operatorname%
{\,\,\text{\bf-}\kern-.98em\DOTSI\intop\ilimits@\!\!}}
\makeatother

\def\\det{\text{det}}

\def\.5{\frac{1}{2}}

\newcommand{\RN}[1]{%
  \textup{\uppercase\expandafter{\romannumeral#1}}%
}

\renewcommand{\epsilon}{\varepsilon}

\newcounter{marnote}

\begin{document}
\title[An extended  Flaherty-Keller formula]{An extended Flaherty-Keller formula for an elastic composite with densely packed convex inclusions }

\author[H.G. Li]{Haigang Li}
\address[H.G. Li]{School of Mathematical Sciences, Beijing Normal University, Laboratory of Mathematics and Complex Systems, Ministry of Education, Beijing 100875, China. }
\email{hgli@bnu.edu.cn}
\thanks{H.G. Li was partially supported by  NSFC (11571042, 11631002, 11971061) and BJNSF (1202013). }

\author[Y. LI]{YAN LI}
\address[Y. LI]{School of Mathematical Sciences, Beijing Normal University, Laboratory of Mathematics and Complex Systems, Ministry of Education, Beijing 100875, China. }
\email{yanli@mail.bnu.edu.cn}
\thanks{}


\date{\today} 



\begin{abstract}
In this paper, we are concerned with the effective elastic property of a two-phase high-contrast periodic composite with densely packed inclusions. The equations of linear elasticity are assumed. We first give a novel proof of the  Flaherty-Keller formula for elliptic inclusions, which improves a recent result of Kang and Yu (Calc.Var.Partial Differential Equations, 2020). We construct an
 auxiliary function consisting of the Keller function and an additional corrected function depending on the coefficients of Lam\'e system and the geometry of  inclusions, to capture the full singular term of the gradient. On the other hand, this  method allows us to deal with the inclusions of arbitrary shape, even with zero curvature. An extended Flaherty-Keller formula is proved for $m$-convex inclusions, $m>2$, curvilinear squares with round off angles, which minimize the elastic modulus under the same volume fraction  of hard inclusions.
\end{abstract}

\maketitle

\section{Introduction and main results}

\subsection{Background and Motivation}
In a two-phase composite where inclusions are close to each other and their strength, such as conductivity and elastic moduli, of the inclusions are of high contrast with that of the matrix, the study of various effective properties of such composite
is an interesting and important topic, because they are always singular. As the distance between inclusions, $\epsilon$, tends to zero, several asymptotic formulae for the effective properties  have been studied, for example, for effective electric or thermal conductivity problem in \cite{K1,K2,LA,Gorb} and Section 10.10 of \cite{G}, for effective shear and extensional modulus in \cite{JJ,HY}.

We suppose  that the fibers are rigid and rather closely packed, so that each fiber nearly touches the ones directly above and below it. For a rectangular array of cylinders in the nearly touching limit (see Figure \ref{Fig.4}), Flaherty and Keller \cite{JJ} obtained two asymptotic formulae for the effective shear modulus and extensional modulus and they also showed the validity numerically when the cylinder inclusions are rigid. Recently, Kang and Yu \cite{HY} gave a mathematically rigorous proof of the Flaherty-Keller formula, with a lower order term $O(\varepsilon^{-1/4})$, based on the primal-dual variational principle. The key of their proof lies in the contribution of test functions to apply the dual principle. However, these singular functions are only valid for two adjacent hard circular inclusions or elliptic inclusions. When the inclusions are of general convex shape, there will be trouble to apply the primal-dual principle, especially, to drive a suitable lower bound.

The contribution of this paper is that we develop another method to overcome this limitation. We construct a family auxiliary functions, containing all geometry information of the inclusions of arbitrary shape. These functions consist of the Keller-type functions and additional corrected functions, stimulated by the idea that we construct the Green function of Laplacian equation. The introduction of such corrected terms is from an important observation. One can regard them as some variants of the basis of the linear space of rigid displacement coupled with the coefficients of Lam\'e system and the local geometric information of the inclusions. On the other hand, this construction can also improve the error term $O(\varepsilon^{-1/4})$ obtained in \cite{HY} to $O(1)$, becoming a bounded term independent of $\varepsilon$, which may be helpful if we do numerical computation.

We would like to point out that these two aspects of achievements above are attributed to our precise gradient estimates, see Proposition \ref{prop1} below. The effective elastic properties (global properties) of a composite are closely related to the stress concentration phenomenon(local properties). When two inclusions with extreme material property are close to touching, the stress blows up in between them. In fact, the dominant contribution to the effective elastic modulus comes from narrow gaps between closely spaced inclusions, while the stress field outside these gaps does not contribute to the leading term of asymptotics of the  effective elastic modulus. So the gradient's local blow-up analysis for Lam\'e system with partially infinite coefficients  has been an important theme in partial differential equation field particularly in last two decades. The analogue in scalar case, where ~$u$~ express the antiplane  displacement, is also called conductivity problem, because these two models are consistent in dimension two. We refer to \cite{BLL1,BLL2,HS,LX, HJL} and references therein for such a development in this topic.

\begin{figure}
  \centering
  \includegraphics[width=0.5\textwidth]{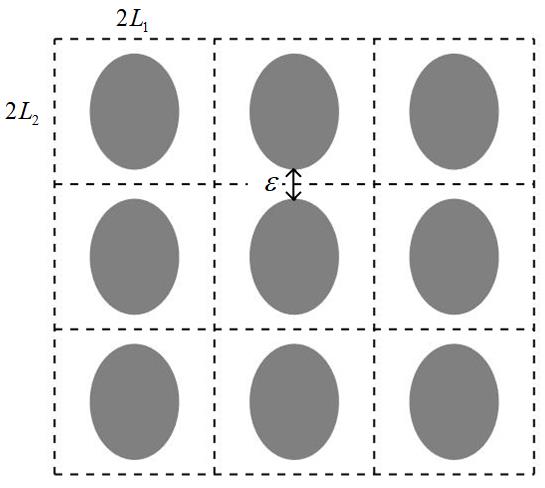}\\
  \caption{Ellipse inclusions}
  \label{Fig.4}
\end{figure}
To formulate our main results precisely, we first describe our domain and notations.
\subsection{Formulation of the problem}
 Here we assume that the composite is spatially periodic, consisting of properly shaped elastic inclusions embedded in an elastic matrix. Let $Y\subset\mathbb{R}^2 $ be a rectangular unit period cell with $2L_{1}$ along the $x_{1}$-axis and $2L_{2}$ along the $x_{2}$-axis, where~$L_{1},~L_{2}\in(0, +\infty)$. Let $D\subset Y$ be a convex domain, with $C^{2}$ boundary  center at the origin and symmetric with respect to $x_{1}$- and $x_{2}$-axes. As in \cite{JJ} we assume that $D$ is close to the horizontal boundary of $Y$,~but away from the vertical boundary. Let $\epsilon/2$ be the distance between $D$ and the horizontal boundary of $Y$, so that the distance between two adjacent inclusions is $\epsilon$.

Assume that $Y\backslash\overline{D}$ is occupied by a homogeneous and isotropic materials with $\mathrm{Lam\acute{e}}$ constants $(\lambda,\mu)$ satisfying the strong ellipticity conditions
$$\mu>0,\quad and\quad \lambda+\mu>0.$$
The elasticity tensors  $\mathbb{C}$  is given by
$$C_{ijkl}=\lambda\delta_{ij}\delta_{kl}+\mu(\delta_{ik}\delta_{jl}+\delta_{il}\delta_{jk}),$$
where $i,j,k,l\in\{1,2\}$ and $\delta_{ij}$ is the Kronecker symbol: $\delta_{ij}=0$ for $i\neq j$, $\delta_{ij}=1$ for $i=j$.

The linear space of rigid displacements in $\mathbb{R}^2$ is
$$\Psi:=\Bigg\{\psi\in C^{1}(\mathbb{R}^2;\mathbb{R}^2)~|~\nabla\psi+(\nabla\psi)^{T}=0\Bigg\},$$
or equivalently
$$\Psi=\mathrm{span}\Bigg\{\psi_{1}=\begin{pmatrix}
1 \\
0
\end{pmatrix},
\psi_{2}=\begin{pmatrix}
0\\
1
\end{pmatrix},
\psi_{3}=\begin{pmatrix}
x_{2}\\
-x_{1}
\end{pmatrix}
\Bigg\}.$$
For a composite with rigid fibers, we consider the following problem  for the $\mathrm{Lam\acute{e}}$ system:
\begin{equation}\label{1.7}
\begin{cases}
\mathcal{L}_{\lambda,\mu}v_{i}=\nabla\cdot\mathbb{C}e(v_{i})=0,&\hbox{in}\ Y\backslash \overline{D},  \\
v_{i}=0,&\hbox{on}\ \partial D,\\
\frac{\partial v_{i}}{\partial \nu_{0}}\Big|_{+}=0,&\hbox{on}\ x_{1}=\pm L_{1},   \quad\quad\quad i=1,2,\\
v_{i}=\pm\frac{1}{2}\psi_{i},&\hbox{on}\ x_{2}=\pm L_{2},
\end{cases}
\end{equation}
 where $v_{i}=\big(v_{i}^{(1)},v_{i}^{(2)} \big)\in H^{1}(Y\backslash \overline{D};\mathbb{R}^2)$, represents the displacement field,
$$e(v_{i})=\frac{1}{2}\Big(\nabla v_{i}+(\nabla v_{i})^{T}\Big)\quad \text{(\textsl{T} for transpose)}$$
is the strain tensor, and the corresponding co-normal derivative on $\partial Y$ is defined by
\begin{align*}
\frac{\partial v_{i}}{\partial \nu_{0}}\Big|_{+}:=\Big(\mathbb{C}e(v_{i})\Big)\vec{n}
=\lambda(\nabla\cdot v_{i})\vec{n}+\mu\Big(\nabla v_{i}+(\nabla v_{i})^{T}\Big)\vec{n},
\end{align*}
and $\vec{n}$ is the  unit outer normal vector of $Y$. Here and throughout this paper the subscript $\pm$ indicates the limit from outside and inside the domain, respectively.

\subsection{The Flaherty-Keller formula}
For a composite with closely spaced rigid fibers the effective shear modulus $\mu^{*}$ and the effective extensional modulus $E^{*}$ are defined as follows (see (2.1) and (2.2) in \cite{JJ})
\begin{align}\label{1.8}
\mu^{*}&=\frac{L_{2}}{L_{1}}\int_{-L_{1}}^{L_{1}}\frac{\partial v_{1}}{\partial \nu_{0}}\Big|_{+}(x_{1},L_{2})\cdot\psi_{1}~dx_{1},
\end{align}
and
\begin{align}
E^{*}&=\dfrac{(1+\rho)(1-2\rho)}{1-\rho}\frac{L_{2}}{L_{1}}\int_{-L_{1}}^{L_{1}}\frac{\partial v_{2}}{\partial \nu_{0}}\Big|_{+}(x_{1},L_{2})\cdot\psi_{2}~dx_{1} \nonumber\\
&=\frac{E}{\lambda+2\mu}\frac{L_{2}}{L_{1}}\int_{-L_{1}}^{L_{1}}\frac{\partial v_{2}}{\partial \nu_{0}}\Big|_{+}(x_{1},L_{2})\cdot\psi_{2}~dx_{1},\label{1.9}
\end{align}
where
$$E=\frac{\mu(3\lambda+2\mu)}{\lambda+\mu},\quad\mbox{and}\quad \rho=\frac{\lambda}{2(\lambda+\mu)}$$
is, respectively, Young's modulus and  Poisson's ratio of the matrix.

Assume that $D$ is an ellipse,
\begin{align}\label{1.2}
\frac{x_{1}^{2}}{a^2}+\frac{x_{2}^{2}}{b^2}\leq1,
\end{align}
where $a$ and $b$ are the length of the short and long semi-axis, respectively, see Figure \ref{Fig.4}. The boundary of $D$ near points $(0,\pm b)$ can be written as, respectively,
\begin{align}\label{1.11}
x_{2}=\pm b\mp\frac{\kappa_{0}}{2}x_{1}^{2}+O(x_{1}^4),
\end{align}
where $\kappa_{0}=b/a^2$ is the curvature of $\partial D$ at the points $(0,\pm b)$. Because the first term in the right hand side of \eqref{1.11} is of order two, we call such elliptic inclusion 2-convex inclusion. The Flaherty-Keller formula is as follows.
\begin{theorem}\label{thm2.1}
(The Flaherty-Keller formula) Let $D$ be as in \eqref{1.2}. Then when $2(L_{2}-b)=:\epsilon\rightarrow0$, the asymptotic formulae for the effective shear modulus as in \eqref{1.8} and the extensional modulus as in \eqref{1.9} satisfy, respectively,
\begin{equation}\label{eq2.2}
\mu_{2}^{*}=\mu\frac{L_{2}}{L_{1}}\frac{\pi}{\sqrt{\kappa_{0}}}\frac{1}{\sqrt{\epsilon}}+O(1),
\end{equation}
and
\begin{equation}\label{eq2.1}
E_{2}^{*}=E\frac{L_{2}}{L_{1}}\frac{\pi}{\sqrt{\kappa_{0}}}\frac{1}{\sqrt{\epsilon}}+O(1),
\end{equation}
where $\kappa_{0}$ is the curvature of $\partial D$ at the points $(0,\pm b)$.
\end{theorem}
\begin{remark}
We would like to remark that Kang and Yu \cite{HY} obtained the error terms are of $O(\varepsilon^{-1/4})$, which are now improved to $O(1)$ in Theorem \ref{thm2.1}. This improvement is due to a suitable construction of the auxiliary  functions, see \eqref{u_1} and \eqref{u_2} below.
\end{remark}

\begin{remark}
As we know, the Keller-type function is not a solution of the Lam\'e system. Although it can be used to capture the first main term of the gradient of $v_{i}$, as in \cite{BLL1,BLL2}, there will be a large error term when we calculate the effective modulus. So in order to prove Theorem \ref{thm2.1}, a novel auxiliary function is needed to seek. Fortunately, a class of corrected functions depending on the Lam\'e parameters $\lambda$ and $\mu$ are constructed, together with the Keller-type function, to overcome this difficulty. Meanwhile, we as well improve the results on gradient estimates previous established in \cite{BLL1,BLL2,HJL}, for more detail see Proposition \ref{prop1}. What is more important, this kind of auxiliary functions  allow us to deal with more general inclusions of arbitrary shape, see Theorem \ref{thm2.2}.
\end{remark}

As an immediate consequence of Theorem \ref{thm2.1}, we have the asymptotics expansion for $\mu_{2}^{*}$ and $E_{2}^{*}$ with respect to the volume fraction, when it is close to its maximum.
For instance, $L_{1}=L_{2}=L~\mbox{and}~a=b=r$,
 the volume fraction  $f_{2}$, occupied by circular inclusions, given by
$$f_{2}=\frac{\pi r^{2}}{4L^{2}},$$
with the maximum $\frac{\pi}{4}$ when the inclusions touch each other.
\begin{corollary}\label{cor 1}
As $\frac{\pi}{4}-f_{2}$ tends to zero, we have the following asymptotic formulae for shear and extensional modulus, respectively,
\begin{align*}
\mu_{2}^{*}=\mu\frac{\pi^{3/2}}{\sqrt{2}}\frac{1}{\sqrt{\frac{\pi}{4}-f_{2}}}+O(1),
\end{align*}
and
\begin{align*}
E_{2}^{*}=E\frac{\pi^{3/2}}{\sqrt{2}}\frac{1}{\sqrt{\frac{\pi}{4}-f_{2}}}+O(1).
\end{align*}
\end{corollary}
\begin{figure}
  \centering
  \includegraphics[width=0.5\textwidth]{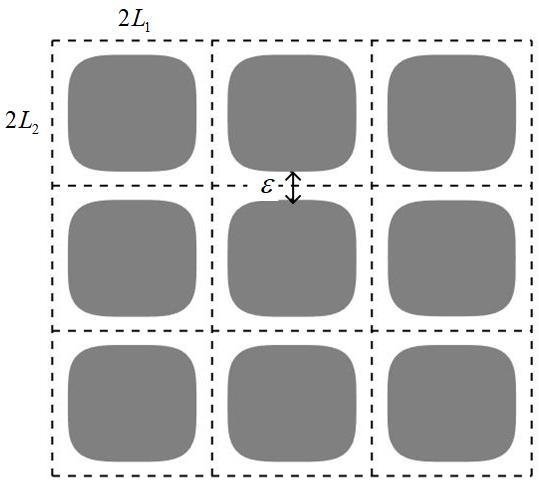}\\
  \caption{$m$-covex inclusions, $m=4$}
  \label{Fig.2}
\end{figure}

\subsection{An extended Flaherty-Keller formula}

The second contribution of this paper is that the method we developed in the proof of Theorem \ref{thm2.1} allows us to deal with more general inclusions. Assume that the inclusion is nearly square (see Figure \ref{Fig.2}), whose boundary is defined as follows:
\begin{align}\label{9.4}
|x_{1}|^m+|x_{2}|^m=r^{m},
\end{align}
where $m>2$ and $r\in \mathbb{R}$ is a half-width of the inclusion. We call the inclusion with such boundary curve as $m$-convex inclusions.

In this subsection we consider the case of $m$-convex inclusions.
 More reason why we study such kind of inclusions will be explained later after Corollary \ref{cor 2}. We take $D$ as the following curvilinear squares with round off angles, with $m>2$,
$$|x_{1}|^{m}+|x_{2}|^{m}\leq r^{m}, $$
where $r$ is a half-width of the inclusion (See Figure \ref{Fig.2}). An extended Flaherty-Keller formula for the effective elastic moduli is as follows:
\begin{theorem}\label{thm2.2}
(An extended Flaherty-Keller formula) Given $m>2$, then, as the distance between two inclusions $\epsilon=2(L_{2}-r)\rightarrow 0$, the asymptotic formulae for the effective shear modulus defined in \eqref{1.8} and the extensional modulus  in \eqref{1.9} are, respectively,
\begin{equation*}
\mu_{m}^{*}=2\mu\frac{L_{2}}{L_{1}}\frac{\pi}{m\sin{\frac{\pi}{m}}}
\frac{1}{\kappa_{0}^{\frac{1}{m}}}\frac{1}{\epsilon^{1-\frac{1}{m}}}+O(1),
\end{equation*}
and
\begin{equation*}
E_{m}^{*}=2E\frac{L_{2}}{L_{1}}\frac{\pi}{m\sin{\frac{\pi}{m}}}
\frac{1}{\kappa_{0}^{\frac{1}{m}}}\frac{1}{\epsilon^{1-\frac{1}{m}}}+O(1),
\end{equation*}
where $\kappa_{0}:=\frac{2}{m}r^{1-m}$.
\end{theorem}

Similarly as in Corollary \ref{cor 1}, let $f_{m}$ be the  volume fraction occupied by curvilinear squares. Then under the assumption that $L_{1}=L_{2}=L$,
$$f_{m}=\frac{r^{2}}{2mL^2}\frac{\Gamma(\frac{1}{m})^2}{\Gamma(\frac{2}{m})},$$
and it has maximum $\frac{\Gamma(\frac{1}{m})^2}{2m\Gamma(\frac{2}{m})}$ when the inclusions touch each other.
\begin{corollary}\label{cor 2}
As $\delta_{m}=\frac{\Gamma(\frac{1}{m})^2}{2m\Gamma(\frac{2}{m})}-f_{m}$ tends to zero, we have
\begin{align*}
\mu_{m}^{*}=\mu\frac{\pi}{\sin\frac{\pi}{m}}
\left(\frac{2}{m^2}\frac{\Gamma(\frac{1}{m})^2}{\Gamma(\frac{2}{m})}\right)^{1-\frac{1}{m}}
\frac{1}{\delta_{m}^{1-\frac{1}{m}}}+O(1)
\end{align*}
and
\begin{align*}
E_{m}^{*}=E\frac{\pi}{\sin\frac{\pi}{m}}\left(\frac{2}{m^2}\frac{\Gamma(\frac{1}{m})^2}{\Gamma(\frac{2}{m})}\right)^{1-\frac{1}{m}}
\frac{1}{\delta_{m}^{1-\frac{1}{m}}}+O(1).
\end{align*}
\end{corollary}

Here we explain the relationship between $m$-convex inclusions and ``Vigdergauz inclusions'', which minimizes the maximum stress concentration in the theory of structure optimization.
This inclusion was first discovered by Vigdergauz in a  series of papers, \cite{V1, V2}. In \cite{YR}, a shape of an optimal inclusion is given in terms of the elliptic integrals of the first kind,  which is called ``Vigdergauz inclusion''.
Indeed, the extremal composite has received a lot of attention, for example, in structural optimization problem \cite{AK93,KS86,M90}.

For the square periodicity cell, given the volume fraction $f$ of the inclusions, the parameter $h(0<h<1)$ depending only on $f$ is the solution of the equation~$f=(1-h)/(1+h)$.
Let the incomplete and complete elliptic integrals of the first kind are, respectively,
$$F(x~|~\mu)=\int_{0}^{x}\frac{ds}{\sqrt{(1-s^2)(1-\mu s^2)}},\quad K(\mu)=F(1~|~\mu),$$
where the parameter $\mu (\frac{1}{2}<\mu<1)$ is a solution of the equation $h=K(1-\mu)/K(\mu)$.
Then the quarter of the boundary of the Vigdergauz inclusion can be given by the following parametrization \cite{YR}:
\begin{align}\label{9.2}
\begin{cases}
x(t)=-\frac{1}{2(1+h)K(\mu)}F(\sqrt{1-t}~|~\mu),&\\
y(t)=\frac{1}{2(1+h)K(\mu)}F(\sqrt{1-\frac{M}{t}}~\Big|~\mu),&
\end{cases}
\end{align}
where the parameter $t\in[M,1]$ and $M=(1-\mu)^2/\mu^2$. It was found in \cite{Gorb} that such inclusions have the nearly square shape. It is very close to an $m$-convex inclusion, under the same volume fraction (see Figure \ref{Fig.1}). So it is also very interesting to consider $m$-convex inclusion, with simple curve boundary \eqref{9.4}, to describe the nearly square shape. We would like to point out an asymptotic formula for the effective conductivity of a composite with $m$-convex inclusion was derived in \cite{Gorb}.

\begin{figure}
  \centering
  \includegraphics[width=0.5\textwidth]{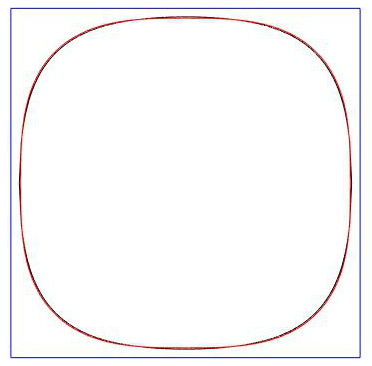}\\
  \caption{A Vigdergauz inclusion (black) and an $m$-convex inclusion (red) with  the same volume fraction. }
  \label{Fig.1}
\end{figure}

As shown above, we have obtained the asymptotic formula for the elastic modulus near the maximum volume fractions at $m=2$ (Corollary \ref{cor 1}) and $m>2$ (Corollary \ref{cor 2}). It is shown that in \cite{YR}  the composites with``Vigdergauz inclusions''  can minimizes the overall energy at a given strain, among all composites made from the same components in the same volume fraction. For this purpose, we compare the elastic modulus of the two cases under the same volume fraction.
\begin{remark}
For circular inclusions and curvilinear square inclusions with $m=4$, for example, when their volume fractions are the same, we can calculate the corresponding elastic moduli of the composite. Namely, for $\delta_{2}:=\frac{\pi}{4}-f_{2}>0$, it is easy to see, from
\begin{align}\label{1.4}
\frac{\pi}{4}-\delta_{2}=f_{2}=f_{4}=\frac{1}{8}\frac{\Gamma(\frac{1}{4})^2}{\Gamma(\frac{1}{2})}-\delta_{4},
\end{align}
that $\delta_{m}>\delta_{2}$. So $\mu_{m}^{*}<\mu_{2}^{*}$ and $E_{m}^{*}<E_{2}^{*}$. For example, taking $\delta_{2}=0.01$, we have
$$ \mu_{2}^{*}\approx  12.53\pi\mu,\quad E_{2}^{*}\approx  12.53\pi E,$$
while by \eqref{1.4}, $\delta_{4}\approx 0.15$, and
$$\mu_{4}^{*}\approx 5.56\pi\mu,\quad E_{4}^{*}\approx 5.56\pi E.$$
This shows that under the  same volume fraction the  elastic modulus at $m=4$ is exactly smaller than  at $m=2$, which is consistent with the conclusion in \cite{YR} that ``Vigdergauz inclusions'' minimize the elasticity modulus.
\end{remark}

\begin{figure}
  \centering
  \hspace{2cm}\includegraphics[width=0.5\textwidth]{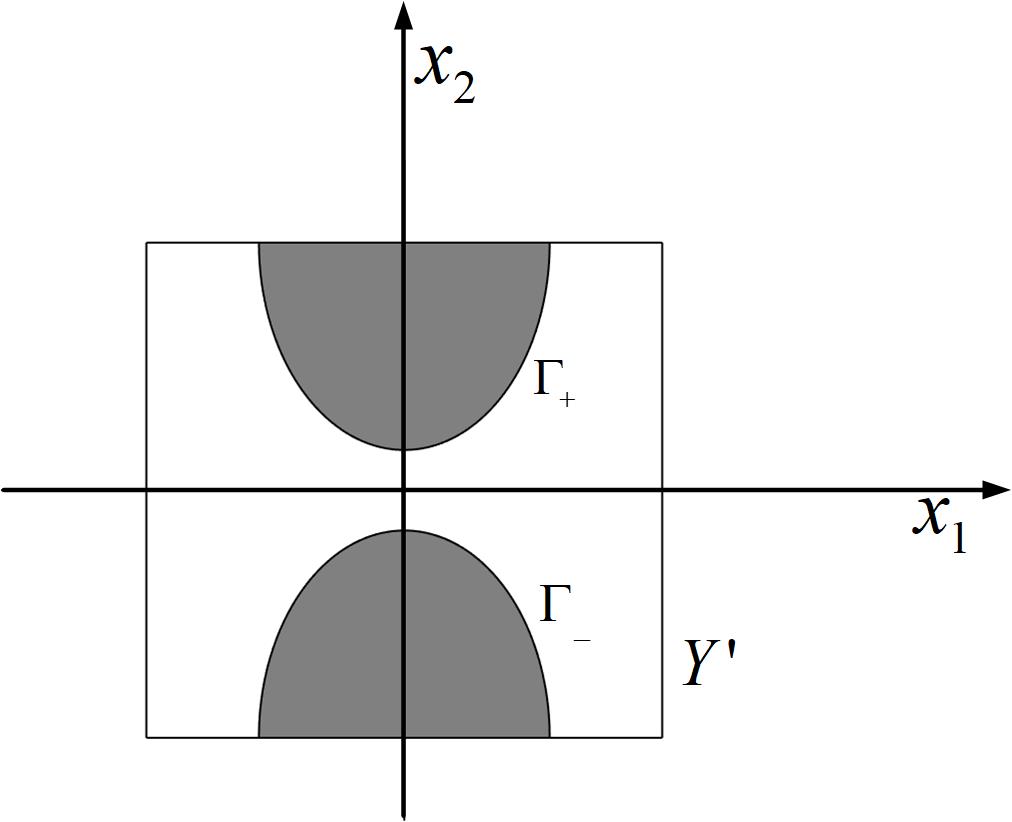}\\
  \caption{Elliptic inclusions}
  \label{Fig.3}
\end{figure}

\subsection{ Outline of the Proof of Theorem \ref{thm2.1}}
We next outline our main idea to prove Theorem\ref{thm2.1}. As in \cite{HY}, we first extend $v_{i}$ to the whole space $\mathbb{R}^{2}$ by periodicity so that the extended function, denoted still by $v_{i}$,  satisfies the following periodic conditions
\begin{equation*}
v_{i}(x_{1},x_{2}+2L_{2})=v_{i}(x_{1},x_{2})+\psi_{i},\quad v_{i}(x_{1}+2L_{1},x_{2})=v_{i}(x_{1},x_{2}).
\end{equation*}
Thus, $e(v_{i})$ is periodic as well. Using $\vec{n}\big|_{x_{2}=L_{2}}=-\vec{n}\big|_{x_{2}=-L_{2}}$ and the boundary condition of $v_{i}$ in \eqref{1.7}, we have
\begin{align*}
&\int_{-L_{1}}^{L_{1}}\frac{\partial v_{i}}{\partial \nu_{0}}\Big|_{+}(x_{1},L_{2})\cdot \psi_{i}\\
=&\int_{-L_{1}}^{L_{1}}\frac{\partial v_{i}}{\partial \nu_{0}}\Big|_{+}(x_{1},-L_{2})\cdot \Big(-\frac{1}{2}\psi_{i}\Big)+\int_{-L_{1}}^{L_{1}}\frac{\partial v_{i}}{\partial \nu_{0}}\Big|_{+}(x_{1},L_{2})\cdot \Big(\frac{1}{2}\psi_{i}\Big)   \\
=&\int_{\partial(Y\backslash \overline{D})}\frac{\partial v_{i}}{\partial \nu_{0}}\Big|_{+}\cdot v_{i}=\int_{Y\backslash \overline{D}}(\mathbb{C}e(v_{i}),e(v_{i}))=:\mathcal{E}_{i}.
\end{align*}
In view of \eqref{1.8} and \eqref{1.9}, the effective moduli $\mu^{*}$ and $E^{*}$ can be expressed in terms of the energy integral, namely,
\begin{equation}\label{3.22}
\mu^{*}=\frac{L_{2}}{L_{1}}\mathcal{E}_{1}\quad\hbox{and}\quad
E^{*}=\dfrac{E}{\lambda+2\mu}\frac{L_{2}}{L_{1}}\mathcal{E}_{2} .
\end{equation}

Now it is more convenient to consider the energy integral $\mathcal{E}_{i}$ in a translated cell $Y_{t}:=Y +(0,L_{2})=(-L_{1},L_{1})\times(0,2L_{2})$. Let us denote $D_{1}=D+(0,2L_{2}),~D_{2}=D$, and $Y'=Y_{t}\backslash\overline{D_{1}\cup D_{2}}$. See Figure \ref{Fig.3} for $Y'$. Set $\Gamma_{-}=(\partial D_{2}\cup \{x_{2}=0\})\cap\partial Y'$ and $\Gamma_{+}=(\partial D_{1}\cup \{x_{2}=2L_{2}\})\cap\partial Y'$.
Still denote $v_{i}$ after translation.
By the periodicity, note that, for any $i=1,2$, $v_{i}|_{Y'}\in H^{1}(Y')$ is the solution to the following problem:
\begin{equation}
\begin{cases}\label{eq3.0}
\mathcal{L}_{\lambda,\mu}v_{i}=\nabla\cdot\mathbb{C}e(v_{i})=0,&\hbox{in}\ Y',  \\
v_{i}=\psi_{i},& \hbox{on}\ \Gamma_{+},\\
v_{i}=0,& \hbox{on}\ \Gamma_{-},\\
\frac{\partial v_{i}}{\partial \nu_{0}}\Big|_{+}=0,&\hbox{on}\ x_{1}=\pm L_{1}.
\end{cases}
\end{equation}
Then,
\begin{align}\label{2.1}
\mathcal{E}_{i}=\int_{Y'}(\mathbb{C}e(v_{i}),e(v_{i}))\,dx.
\end{align}

Denote the two  points on $\partial D_{1}$ and $\partial D_{2}$, achieving the distance between $D_{1}$ and $D_{2}$,
$$P_{1}=\Big(0,\frac{\epsilon}{2}\Big)\in\partial D_{1}\quad\hbox{and}\quad P_{2}=\left(0,-\frac{\epsilon}{2}\right)\in\partial D_{2}.$$
Then the parts of $\partial D_{1}$ and $\partial D_{2}$ near $P_{1}$ and $P_{2}$, respectively, can be represented as follows
\begin{align*}
x_{2}&=\frac{\epsilon}{2}+h_{1}(x_{1})=\frac{\epsilon}{2}-\frac{\kappa_{0}}{2}x_{1}^{2}+O(x_{1}^4),\\ x_{2}&=-\frac{\epsilon}{2}+h_{2}(x_{1})=-\frac{\epsilon}{2}+\frac{\kappa_{0}}{2}x_{1}^{2}+O(x_{1}^4), .
\end{align*}
for $|x_{1}|\leq a$. We always use $\delta(x_{1})$ to denote the vertical distance between the inclusions,
$$\delta(x_{1}):=\epsilon+h_{1}(x_{1})-h_{2}(x_{1}),\quad\mbox{for}\, |x_{1}|\leq a$$
and for $0\leq s\leq r$,
$$\Omega_s:=\left\{x=(x_{1},x_{2})\in \mathbb{R}^2~\big|~ -\frac{\varepsilon}{2}+h_2(x_{1})<x_{2}<\frac{\varepsilon}{2}+h_1(x_{1}), ~|x_{1}|<s \right\}.$$

Because, as mentioned before, the greatest stress occurs in the narrow gaps between $D_{1}$ and $D_{2}$ while the outside stress does not contribute the singularity, we now construct two auxiliary functions $u_{i}\in C^{2}(\mathbb{R}^2)$, such that
  $u_{i}=\Psi_{i}~ \hbox{on} ~\Gamma_{+}, u_{i}=0~\hbox{on}~ \Gamma_{-},~\frac{\partial u_{i}}{\partial \nu_{0}}\Big|_{+}=0~\hbox{on}~x_{1}=\pm L_{1}$, and for $x\in \Omega_{\frac{r}{2}}$,
\begin{equation}\label{u_1}
u_{1}:=
\bar{u}_{1}+\tilde{u}_{1}:=\frac{2x_{2}+\delta(x_{1})}{2\delta(x_{1})}\begin{pmatrix}
1\\
0
\end{pmatrix}+\Big(\big(\frac{x_{2}}{\delta(x_{1})}\big)^2-\frac{1}{4}~\Big)\begin{pmatrix}
(2-\frac{\mu}{\lambda+2\mu})\frac{\kappa_{0}}{3}x_{2} \\\\
(1-\frac{\mu}{\lambda+2\mu})\kappa_{0}x_{1}
\end{pmatrix},
\end{equation}
\begin{equation}\label{u_2}u_{2}:=\bar{u}_{2}+\tilde{u}_{2}:=\frac{2x_{2}+\delta(x_{1})}{2\delta(x_{1})}\begin{pmatrix}
0\\
1
\end{pmatrix}+\Big(\big(\frac{x_{2}}{\delta(x_{1})}\big)^2-\frac{1}{4}~\Big)\begin{pmatrix}
\frac{\lambda+\mu}{\mu}\kappa_{0}x_{1} \\\\
-\frac{\lambda}{3\mu}\kappa_{0}x_{2}
\end{pmatrix},\quad\quad\quad
\end{equation}
where $\kappa_{0}=1/r$, and
\begin{align}
\|u_{i}\|_{C^{2}(\mathbb{R}^{2}\backslash\Omega_{\frac{r}{2}})}\leq C.\label{3.u}
\end{align}

We note that the parts $\tilde{u}_{i}$ can be regarded as  variants of $\psi_{3}=(x_{2},-x_{1})^{T}$, but they also depend on the coefficients of Lam\'e system. Then we can use an adapted version the energy iteration technique developed  in \cite{BLL1, BLL2},  together with the rescaling argument, $W^{1,p}$ estimates, and Sobolev embedding theorem,  to obtain the following improved estimates.
\begin{prop}\label{prop1}
For $i=1,2,$ we have
\begin{equation}\label{3.68}
\left|\nabla (v_{i}- u_{i})\right|\leq C,
\end{equation}
where $C$ is independent of $\epsilon$.

Consequently,
\begin{align}\label{1.10}
\nabla v_{i}=\nabla \bar{u}_{i}+\nabla\tilde{u}_{i}+O(1),\quad i=1,2.
\end{align}
\end{prop}
\begin{remark}
We remark that \eqref{1.10} is an improvement of the results  in \cite{BLL1,BJL}, where the lower and upper bounds of $|\nabla v_{i}^{\alpha}|$, $i,\alpha=1,2$ are obtained. While \eqref{1.10} captures the full singular terms of $\nabla v_{i}$.  It is because our novel constructions of $\tilde{u}_{i}$ that we can use Proposition \ref{prop1} to prove Theorem \ref{thm2.1}. In particular, it makes the error term be improved to the order of $O(1)$. On the other hand, this construction of $u_{i}$ allows us to study $m$-convex inclusions, even they have zero curvature when $m>2$. It can also be used to deal with more general convex inclusions, see\cite{LX}.
 \end{remark}
 For $m$-convex inclusions, we suppose that
 \begin{equation*}
h_{1}(x_{1})=\frac{\kappa_{0}}{2}|x_{1}|^{m}+O(|x_{1}|^{m})\quad\hbox{and}\quad h_{2}(x_{1})=-\frac{\kappa_{0}}{2}|x_{1}|^{m}+O(|x_{1}|^{m}),\quad m>2 ,
\end{equation*}
where $\kappa_{0}=2r^{1-m}/m$. Instead of \eqref{u_1} and \eqref{u_2}, by constructing the auxiliary functions $u_{i}\in C^{2}(\mathbb{R}^2)$  in $\Omega_{\frac{r}{2}}$,
\begin{align*}
u_{1}=&\bar{u}_{1}+\tilde{u}_{1}:=\frac{2x_{2}+\delta(x_{1})}{2\delta(x_{1})}\begin{pmatrix}
1 \\
0
\end{pmatrix}+\Big(\big(\frac{x_{2}}{\delta(x_{1})}\big)^2-\frac{1}{4}~\Big)\begin{pmatrix}
(2-\frac{\mu}{\lambda+2\mu})\frac{\kappa_{0}}{3}\frac{m(m-1)}{2}x_{2}x_{1}^{m-2} \\\\
(1-\frac{\mu}{\lambda+2\mu})\kappa_{0}\frac{m}{2}x_{1}^{m-1}
\end{pmatrix}
\end{align*}
and
\begin{align*}
u_{2}=&\bar{u}_{2}+\tilde{u}_{2}:=\frac{2x_{2}+\delta(x_{1})}{2\delta(x_{1})}\begin{pmatrix}
0\\
1
\end{pmatrix}+\Big(\big(\frac{x_{2}}{\delta(x_{1})}\big)^2-\frac{1}{4}~\Big)\begin{pmatrix}
\frac{\lambda+\mu}{\mu}\kappa_{0}\frac{m}{2}x_{1}^{m-1} \\\\
-\frac{\lambda}{3\mu}\kappa_{0}\frac{ m(m-1)}{2}x_{2}x_{1}^{m-2}
\end{pmatrix},
\end{align*}
where $\delta(x_{1})=\epsilon+h_{1}(x_{1})-h_{2}(x_{1})$, we can prove the extended Flaherty-Keller formula in Theorem \ref{thm2.2}.

The rest of this paper is organized as follows.  We first present some elementary calculations of the auxiliary functions, constructed in \eqref{u_1} and \eqref{u_2}, then use them to prove Proposition \ref{prop1}, finally give a new proof of the Flaherty-Keller formula in Theorem \ref{thm2.1}. By this method, the extended Flaherty-Keller formula is proved in section 3, with the main differences provided.

\section{Proof of Theorem \ref{thm2.1}}

This section is devoted to proving Theorem \ref{thm2.1}. We first reduce its proof to the asymptotic formula of $\mathcal{E}_{i}$, Theorem \ref{thm3.1} below, then we construct an auxiliary function, which depends on the Lam\'{e} system to capture the main terms up to $O(1)$. Finally, we use $\nabla v_{i}$'s asymptotics to prove Theorem \ref{thm2.1}.

Throughout the paper, unless otherwise stated, we use $C$ to denote some positive constant, whose values may vary from line to line, depending only on $a,~b,~r$, and an upper bound of the $C^{2}$, norms of $\partial D_{1}$, $\partial D_{2}$ and $\partial Y'$, but not on $\epsilon$. We call a constant having such dependence a universal constant.
First, by the standard theory for elliptic systems, we have
$$\|\nabla v_{i}\|_{L^{\infty}(Y'\backslash \Omega_{\frac{r}{2}})}\leq C,\quad i=1,2.$$
It follows that
\begin{equation}\label{eq3.6}
\int_{Y'\backslash\Omega_{\frac{r}{2}}}(\mathbb{C}e(v_{i}),e(v_{i}))\,dx\leq C,\quad i=1,2.
\end{equation}
Thus, in the following we only need to deal with the  integrals in $\Omega_{\frac{r}{2}}$.
For readers' convenience, in what follows we assume $a=b=r$, and
\begin{equation*}
h_{1}(x_{1})=\frac{\kappa_{0}}{2}x_{1}^{2}\quad\hbox{and}\quad h_{2}(x_{1})=-\frac{\kappa_{0}}{2}x_{1}^{2},
\end{equation*}
omitting the term $O(x_{1}^4)$, where $\kappa_{0}=1/r$ and
$$\delta (x_{1})=\epsilon+\kappa_{0}x_{1}^{2}.$$

We have the following conclusion.
\begin{theorem}\label{thm3.1}
The energy integral $\mathcal{E}_{1}$ and $\mathcal{E}_{2}$, defined in \eqref{2.1}, have the following expansion, namely
\begin{align}
\mathcal{E}_{1}=\frac{\pi\mu}{\sqrt{\kappa_{0}}}\frac{1}{\sqrt{\epsilon}}+O(1)\quad\mbox{and}\quad
\mathcal{E}_{2}=\frac{\pi(\lambda+2\mu)}{\sqrt{\kappa_{0}}}\frac{1}{\sqrt{\epsilon}}+O(1),
\end{align}
where $\kappa_{0}=1/r$, as $\epsilon\rightarrow 0.$
\end{theorem}

It is clear that Theorem \ref{thm2.1} is an immediate consequence of Theorem \ref{thm2.1}.

\begin{proof}[Proof of Theorem \ref{thm2.1}]
Recalling \eqref{3.22}, we have the effective elastic modulus
$$\mu^{*}=\mu\frac{L_{2}}{L_{1}}\frac{\pi}{\sqrt{\kappa_{0}}}\frac{1}{\sqrt{\epsilon}}+O(1)$$
and
\begin{align*}
E^{*}&=\dfrac{E}{\lambda+2\mu}\frac{L_{2}}{L_{1}}\frac{(\lambda+2\mu)\pi}{\sqrt{\kappa_{0}}}\frac{1}{\sqrt{\epsilon}}+O(1)=E\frac{L_{2}}{L_{1}}\frac{\pi}{\sqrt{\kappa_{0}}}\frac{1}{\sqrt{\epsilon}}+O(1),
\end{align*}
as $\epsilon\rightarrow 0.$
This completes the proof of Theorem\ref{thm2.1}.
\end{proof}

In what follows we will use Proposition \ref{prop1} to prove Theorem \ref{thm3.1}. To this end, we first give some elementary estimates.
\subsection{Some elementary estimates.}
A direct calculation gives the first order derivatives of $u_{1}$, defined by \eqref{u_1},
\begin{align}\label{3.26}
\partial_{x_{1}}\bar{u}_{1}^{(1)}=-2\kappa_{0}\frac{x_{1}x_{2}}{\delta^2(x_{1})},\quad
\partial_{x_{2}}\bar{u}_{1}^{(1)}&=\frac{1}{\delta(x_{1})},\quad
\partial_{x_{2}}\tilde{u}_{1}^{(2)}=\frac{2(\lambda+\mu)\kappa_{0}}{\lambda+2\mu}\frac{x_{1}x_{2}}{\delta^2(x_{1})},
\end{align}
and the following estimates for other terms
\begin{align}
&|\partial_{x_{1}}\bar{u}_{1}^{(1)}|,\,|\partial_{x_{2}}\tilde{u}_{1}^{(2)}|\leq \frac{C|x_{1}|}{\delta(x_{1})},\quad\quad\quad|\partial_{x_{1}}\tilde{u}_{1}^{(1)}|,\,|\partial_{x_{2}}\tilde{u}_{1}^{(1)}|,\,
|\partial_{x_{1}}\tilde{u}_{1}^{(2)}|\leq C.\label{3.11}
\end{align}
Further, for second order derivatives, we have
\begin{align}
\partial_{x_{1}x_{1}}\bar{u}_{1}^{(1)}&=-2\kappa_{0}\frac{x_{2}}{\delta^2(x_{1})}+\mathcal{R}_{11}^{11},\quad \partial_{x_{1}x_{2}}\bar{u}_{1}^{(1)}=-2\kappa_{0}\frac{x_{1}}{\delta^2(x_{1})}, \label{u11_11}\\
\partial_{x_{2}x_{2}}\tilde{u}_{1}^{(1)}&=\frac{2(2\lambda+3\mu)\kappa_{0}}{\lambda+2\mu}\frac{x_{2}}{\delta^2(x_{1})},\quad
\partial_{x_{1}x_{2}}\tilde{u}_{1}^{(2)}=\frac{2(\lambda+\mu)\kappa_{0}}{\lambda+2\mu}\frac{x_{2}}{\delta^2(x_{1})}
+\mathcal{R}_{12}^{12},\label{u11_21}\\
\partial_{x_{2}x_{2}}\tilde{u}_{1}^{(2)}&=\frac{2(\lambda+\mu)\kappa_{0}}{\lambda+2\mu}\frac{x_{1}}{\delta^2(x_{1})},\label{3.49}
\end{align}
where
\begin{align}\label{3.e}
\mathcal{R}_{11}^{11}=8\kappa_{0}^2\frac{x_{1}^2x_{2}}{\delta^3(x_{1})},
\quad\quad \mathcal{R}_{12}^{12}=\frac{-8\kappa_{0}^2(\lambda+\mu)}{\lambda+2\mu}\frac{x_{1}^{2}x_{2}}{\delta^3(x_{1})}.
\end{align}
It is clear that
\begin{align}
&\left|\mathcal{R}_{11}^{11}\right|,\,\left|\mathcal{R}_{12}^{12}\right|\leq\frac{C}{\delta(x_{1})},\quad
\left|\partial_{x_{1}x_{1}}\tilde{u}_{1}^{(1)}\right|\leq C,
\label{3.97}\\
&\left|\partial_{x_{1}x_{2}}\tilde{u}_{1}^{(1)}\right|,\,
\left|\partial_{x_{1}x_{1}}\tilde{u}_{1}^{(2)}\right|\leq\frac{C|x_{1}|}{\delta(x_{1})}.\label{3.82}
\end{align}

Recalling $u_{1}=\bar{u}_{1}+\tilde{u}_{1}$ and \eqref{u11_11}, \eqref{u11_21} and \eqref{3.49}, we have
$$(\lambda+2\mu)(\partial_{x_{1}x_{1}}\bar{u}_{1}^{(1)}-\mathcal{R}_{11}^{11})
+\mu\partial_{x_{2}x_{2}}\tilde{u}_{1}^{(1)}
+(\lambda+\mu)(\partial_{x_{1}x_{2}}\tilde{u}_{1}^{(2)}-\mathcal{R}_{12}^{12})=0.$$
So,
\begin{align}
(\mathcal{L}_{\lambda,\mu}u_{1})^{(1)}&=\mu\triangle u_{1}^{(1)}+(\lambda+\mu)\Big(\partial_{x_{1}x_{1}}u_{1}^{{(1)}}
+\partial_{x_{1}x_{2}}u_{1}^{(2)}\Big)\nonumber\\
&=(\lambda+2\mu)\partial_{x_{1}x_{1}}\tilde{u}_{1}^{(1)}+(\lambda+2\mu)\mathcal{R}_{11}^{11}
+(\lambda+\mu)\mathcal{R}_{12}^{12}. \label{3.b}
\end{align}
By \eqref{3.97}, we obtain
\begin{align}\label{3.a12}
\left|(\mathcal{L}_{\lambda,\mu}u_{1})^{(1)}\right|\leq C\Big(\frac{1}{\delta(x_{1})}+1\Big).
\end{align}
On the other hand, using \eqref{u11_11}, \eqref{u11_21} and \eqref{3.49} again, we have
$$(\lambda+2\mu)\partial_{x_{2}x_{2}}\tilde{u}_{1}^{(2)}
+(\lambda+\mu)\partial_{x_{2}x_{1}}\bar{u}_{1}^{(1)}=0.$$
Thus,
\begin{align}
(\mathcal{L}_{\lambda,\mu}u_{1})^{(2)}&=\mu\triangle u_{1}^{(2)}+(\lambda+\mu)\Big(\partial_{x_{2}x_{1}}u_{1}^{{(1)}}
+\partial_{x_{2}x_{2}}u_{1}^{(2)}\Big)\nonumber\\
&=\mu\partial_{x_{1}x_{1}}\tilde{u}_{1}^{(2)}
+(\lambda+\mu)\partial_{x_{2}x_{1}}\tilde{u}_{1}^{(1)}.\label{3.d}
\end{align}
Using \eqref{3.82},
\begin{align}\label{3.a12r}
\left|(\mathcal{L}_{\lambda,\mu}u_{1})^{(2)}\right|
\leq \frac{C|x_{1}|}{\delta(x_{1})}.
\end{align}
This, together with \eqref{3.a12}, yields
\begin{equation}\label{Lu_1}
\Big|\mathcal{L}_{\lambda,\mu}u_{1}\Big|
\leq\left|(\mathcal{L}_{\lambda,\mu}u_{1})^{(1)}\right|
+\left|(\mathcal{L}_{\lambda,\mu}u_{1})^{(2)}\right|\leq\frac{C}{\delta(x_{1})}.
\end{equation}

Similarly, for $i=2$, a direct calculation gives the first order derivatives of $u_{2}$, defined by \eqref{u_2},
\begin{align} \label{3.35}
\partial_{x_{1}}\bar{u}_{2}^{(2)}=-2\kappa_{0}\frac{x_{1}x_{2}}{\delta^2(x_{1})},\quad
\partial_{x_{2}}\bar{u}_{2}^{(2)}&=\frac{1}{\delta(x_{1})},\quad
\partial_{x_{2}}\tilde{u}_{2}^{(1)}=\frac{2(\lambda+\mu)\kappa_{0}}{\mu}\frac{x_{1}x_{2}}{\delta^2(x_{1})}.
\end{align}
It is easy to see that
\begin{align}
&\left|\partial_{x_{1}}\bar{u}_{2}^{(2)}\right|,\,\left|\partial_{x_{2}}\tilde{u}_{2}^{(1)}\right|\leq\frac{C|x_{1}|}{\delta(x_{1})},\quad
\quad\quad\left|\partial_{x_{1}}\tilde{u}_{2}^{(1)}\right|,\,\left|\partial_{x_{1}}\tilde{u}_{2}^{(2)}\right|
,\,\left|\partial_{x_{2}}\tilde{u}_{2}^{(2)}\right|\leq C.\label{3.40}
\end{align}
Further,
\begin{align}
\partial_{x_{1}x_{1}}\bar{u}_{2}^{(2)}&=-2\kappa_{0}\frac{x_{2}}{\delta^2(x_{1})}
+\mathcal{R}_{22}^{11},\quad
\partial_{x_{1}x_{2}}\bar{u}_{2}^{(2)}=-2\kappa_{0}\frac{x_{1}}{\delta^2(x_{1})},\label{3.86}\\
\partial_{x_{1}x_{2}}\tilde{u}_{2}^{(1)}&=\frac{2\kappa_{0}(\lambda+\mu)}{\mu}\frac{x_{2}}{\delta^2(x_{1})}
+\mathcal{R}_{21}^{12},\quad
\partial_{x_{2}x_{2}}\tilde{u}_{2}^{(1)}=\frac{2\kappa_{0}(\lambda+\mu)}{\mu}\frac{x_{1}}{\delta^2(x_{1})},
\label{3.87}\\
\partial_{x_{2}x_{2}}\tilde{u}_{2}^{(2)}&=\frac{-2\lambda\kappa_{0}}{\mu}\frac{x_{2}}{\delta^2(x_{1})},\label{3.88}
\end{align}
where
\begin{align}\label{3.89}
\mathcal{R}_{22}^{11}=8\kappa_{0}^2\frac{x_{1}^2x_{2}}{\delta^3(x_{1})},\quad
\mathcal{R}_{21}^{12}=\frac{-8\kappa_{0}^2(\lambda+\mu)}{\mu}\frac{x_{1}^2x_{2}}{\delta^3(x_{1})}.
\end{align}
It is clear that
\begin{align}
&\left|\mathcal{R}_{22}^{11}\right|,\,\left|\mathcal{R}_{21}^{12}\right|\leq\frac{C}{\delta(x_{1})}
,\quad\quad
\left|\partial_{x_{1}x_{1}}\tilde{u}_{2}^{(2)}\right|\leq C,\label{3.43}\\
&\left|\partial_{x_{1}x_{1}}\tilde{u}_{2}^{(1)}\right|,\,
\left|\partial_{x_{1}x_{2}}\tilde{u}_{2}^{(2)}\right|\leq \frac{C|x_{1}|}{\delta(x_{1})}.
\label{3.44}
\end{align}

Recalling $u_{2}=\bar{u}_{2}+\tilde{u}_{2}$ and  \eqref{3.86}, \eqref{3.87} and \eqref{3.88}, we have
\begin{align*}
(\lambda+\mu)\partial_{x_{1}x_{2}}\bar{u}_{2}^{(2)}+\mu\partial_{x_{2}x_{2}}\tilde{u}_{2}^{(1)}=0.
\end{align*}
Thus
\begin{align}
(\mathcal{L}_{\lambda,\mu}u_{2})^{(1)}&=\mu\triangle u_{2}^{(1)}+(\lambda+\mu)\Big(\partial_{x_{1}x_{1}}u_{2}^{{(1)}}+\partial_{x_{1}x_{2}}u_{2}^{(2)}\Big)\nonumber\\
&=(\lambda+2\mu)\partial_{x_{1}x_{1}}\tilde{u}_{2}^{(1)}+(\lambda+\mu)\partial_{x_{1}x_{2}}\tilde{u}_{2}^{(2)}.\label{3.107}
\end{align}
This, combining with \eqref{3.44}, yields
\begin{align}\label{3.f}
\left|(\mathcal{L}_{\lambda,\mu}u_{2})^{(1)}\right|\leq\frac{C|x_{1}|}{\delta(x_{1})}.
\end{align}
By the same way, using \eqref{3.86}, \eqref{3.87} and \eqref{3.88}, we have
$$(\lambda+2\mu)\partial_{x_{2}x_{2}}\tilde{u}_{2}^{(2)}
+\mu(\partial_{x_{1}x_{1}}\bar{u}_{2}^{(2)}-\mathcal{R}_{22}^{11})
+(\lambda+\mu)(\partial_{x_{1}x_{2}}\tilde{u}_{2}^{(1)}-\mathcal{R}_{21}^{12})=0.
$$
So
\begin{align}
(\mathcal{L}_{\lambda,\mu}u_{2})^{(2)}&=\mu\triangle u_{2}^{(2)}+(\lambda+\mu)\Big(\partial_{x_{2}x_{2}}u_{2}^{{(1)}}+\partial_{x_{2}x_{1}}u_{2}^{(2)}\Big)\nonumber\\
&=\mu\partial_{x_{1}x_{1}}\tilde{u}_{2}^{(2)}+
\mu\mathcal{R}_{22}^{11}+(\lambda+\mu)\mathcal{R}_{21}^{12},\label{3.a10}
\end{align}
combining with \eqref{3.43}, yields
\begin{align}\label{3.m}
\left|(\mathcal{L}_{\lambda,\mu}u_{2})^{(2)}\right|\leq C(\frac{1}{\delta(x_{1})}+1).
\end{align}
Therefore, we have
\begin{equation}\label{Lu_2}
\left|\mathcal{L}_{\lambda,\mu}u_{2}\right|
\leq \left|(\mathcal{L}_{\lambda,\mu}u_{2})^{(1)}\right|
+\left|(\mathcal{L}_{\lambda,\mu}u_{2})^{(2)}\right|\leq\frac{C}{\delta(x_{1})}.
\end{equation}
We remark that  estimates \eqref{Lu_1} and \eqref{Lu_2} will improve the gradient estimates obtained in \cite{BLL1}.

\subsection{Proof of Proposition \ref{prop1}}
 For $i=1,2$, let $w_{i}:=v_{i}-u_{i}$. Thus $w_{i}$ is the solution to the following problem
\begin{equation}
\begin{cases}\label{3.25}
\mathcal{L}_{\lambda,\mu}w_{i}=-\mathcal{L}_{\lambda,\mu}u_{i},&\hbox{in}\ Y',  \\
w_{i}=0,&\hbox{on}\ \Gamma_{+},\\
w_{i}=0,&\hbox{on}\ \Gamma_{-},\\
\frac{\partial w_{i}}{\partial \nu_{0}}\Big|_{+}=0,&\hbox{on}\ x_{1}=\pm L_{1}.
\end{cases}
\end{equation}

In order to prove Proposition \ref{prop1}, we only need to prove the order of $|\nabla w_{i}|$ is $O(1)$. The following two Lemmas are needed. The first one is to show that the global energy of $w_{i}$ is bounded.
\begin{lemma}\label{lem3.1}
For $i=1,2$, the energy of $w_{i}$ on $Y'$ is bounded by $C$, that is,
\begin{equation}\label{3.24}
\int_{Y'}|\nabla w_{i}|^2dx\leq C.
\end{equation}
\end{lemma}
\begin{proof}
For case $i=1$.

From (3.25) of \cite{BLL1}, there exists $r_{0}\in(r/4,r/3)$ such that
\begin{align}\label{3.53}
\int_{\substack{|x_{1}|=r_{0},\\-\epsilon/2+h_{2}(x_{1})<x_{2}<\epsilon/2+h_{1}(x_{1})}}|w_{1}|\,dx_{2}
&\leq C\left(\int_{Y'}|\nabla w_{1}|^2\,dx\right)^{1/2},
\end{align}
and by (3.26) in \cite{BLL1}, we have
\begin{align}\label{3.57}
\int_{Y'}|\nabla w_{1}|^2\,dx
&\nonumber\leq C\left(\left|\int_{\Omega_{r_{0}}}w_{1}^{(1)}\left(\mathcal{L}_{\lambda,\mu}u_{1}\right)^{(1)}\,dx\right|
+\left|\int_{\Omega_{r_{0}}}w_{1}^{(2)}\left(\mathcal{L}_{\lambda,\mu}u_{1}\right)^{2}\,dx\right|\right)\\
&\qquad+C\left(\int_{Y'\backslash\Omega_{r_{0}}}|\nabla w_{1}|^2\,dx\right)^{1/2}.
\end{align}
For the first term in the right hand side of \eqref{3.57},
to use integration by parts for \eqref{3.b}, recalling \eqref{3.e} we introduce two functions
\begin{align}\label{3.101}
\mathcal{T}_{11}^{11}=4\kappa_{0}^2\frac{x_{1}^2x_{2}^{2}}{\delta^3(x_{1})},
\quad\mathcal{T}_{12}^{12}=\frac{-4\kappa_{0}^2(\lambda+\mu)}{\lambda+2\mu}\frac{x_{1}^2x_{2}^{2}}{\delta^3(x_{1})},
\end{align}
such that
$$\partial_{x_{2}}\mathcal{T}_{11}^{11}=\mathcal{R}_{11}^{11},
\quad\quad\partial_{x_{2}}\mathcal{T}_{12}^{12}=\mathcal{R}_{12}^{12}.$$
Notice that
\begin{align}\label{3.84}
\left|\mathcal{T}_{11}^{11}\right|&\leq C,\quad \left|\mathcal{T}_{12}^{12}\right|\leq C.
\end{align}
Thus,
\begin{align}
&\left|\int_{\Omega_{r_{0}}}w_{1}^{(1)}\Big((\lambda+2\mu)\mathcal{R}_{11}^{11}
+(\lambda+\mu)\mathcal{R}_{12}^{12}\Big)\,dx\right| \nonumber\\
&=\,\left|\int_{\Omega_{r_{0}}}w_{1}^{(1)}\partial_{x_{2}}\Big((\lambda+2\mu)\mathcal{T}_{11}^{11}
+(\lambda+\mu)\mathcal{T}_{12}^{12}\Big)\,dx\right|\nonumber\\
&=\,\left|-\int_{\Omega_{r_{0}}}\partial_{x_{2}}w_{1}^{(1)}\Big((\lambda+2\mu)\mathcal{T}_{11}^{11}
+(\lambda+\mu)\mathcal{T}_{12}^{12}\Big)\,dx\right|\nonumber\\ \label{3.100}
&\leq\, C\left(\int_{Y'}\left|\nabla w_{1}\right|^2\,dx\right)^{1/2}.
\end{align}

By \eqref{3.11}, we have
\begin{align}\label{3.h}
\int_{\Omega_{r_{0}}}\left|\partial_{x_{1}}\tilde{u}_{1}^{(1)}\right|^2\,dx \leq C.
\end{align}
Combining with \eqref{3.53}, we obtain
\begin{align}\label{3.99}
&\left|\int_{\Omega_{r_{0}}}w_{1}^{(1)} \partial_{x_{1}x_{1}}\tilde{u}_{1}^{(1)}\,dx\right|\nonumber\\
\leq &\left|-\int_{\Omega_{r_{0}}}\partial_{x_{1}}w_{1}^{(1)} \partial_{x_{1}}\tilde{u}_{1}^{(1)}\,dx\right|
+\left|\int_{\substack{|x_{1}|=r_{0},\\-\epsilon/2+h_{2}(x_{1})<x_{2}<\epsilon/2+h_{1}(x_{1})}}w_{1}^{(1)}
\partial_{x_{1}}\tilde{u}_{1}^{(1)}\,dx_{2}\right| \nonumber\\
\leq & C\left(\int_{Y'}\left|\nabla w_{1}\right|^2\,dx\right)^{1/2}.
\end{align}
Thus, recalling \eqref{3.b}, we have
\begin{align}\label{3.98}
\left|\int_{\Omega_{r_{0}}}w_{1}^{(1)}\left(\mathcal{L}_{\lambda,\mu}u_{1}\right)^{(1)}\,dx\right|
\leq C\left(\int_{Y'}\left|\nabla w_{1}\right|^2\,dx\right)^{1/2}.
\end{align}

By using \eqref{3.11},
\begin{align}\label{3.s}
\int_{\Omega_{r_{0}}}\left|\partial_{x_{1}}\tilde{u}_{1}^{(2)}\right|^2\,dx
\leq C.
\end{align}
Similar to \eqref{3.99}, combining with \eqref{3.53}, we  obtain
\begin{align}
&\left|\int_{\Omega_{r_{0}}}w_{1}^{(2)}\Big(\partial_{x_{1}x_{1}}\tilde{u}_{1}^{(2)}\Big)\,dx\right|
\leq  C\left(\int_{Y'}\left|\nabla w_{1}\right|^2\,dx\right)^{1/2}.\label{3.i}
\end{align}
In view of  \eqref{3.h}, we have
\begin{align}
\left|\int_{\Omega_{r_{0}}}w_{1}^{(1)}\Big(\partial_{x_{2}x_{1}}\tilde{u}_{1}^{(1)}\Big)\,dx\right|
&=\left|-\int_{\Omega_{r_{0}}}\partial_{x_{2}}w_{1}^{(2)}\partial_{x_{1}}\tilde{u}_{1}^{(1)}\,dx\right| \nonumber\\
\leq& \left(\int_{\Omega_{r_{0}}}\left|\partial_{x_{1}}\tilde{u}_{1}^{(1)}\right|^2\,dx\right)^{1/2}
\left(\int_{Y'}\left|\nabla w_{1}\right|^2\,dx\right)^{1/2} \nonumber\\
\leq& C\left(\int_{Y'}\left|\nabla w_{1}\right|^2\,dx\right)^{1/2}.\label{3.j}
\end{align}
Thus, recalling \eqref{3.d},
\begin{align}\label{3.59}
&\left|\int_{\Omega_{r_{0}}}w_{1}^{(2)}\left(\mathcal{L}_{\lambda,\mu}u_{1}\right)^{(2)}\,dx\right|
\leq C\left(\int_{Y'}\left|\nabla w_{1}\right|^2\,dx\right)^{1/2}.
\end{align}
By \eqref{3.57}, \eqref{3.98} and \eqref{3.59}, we have
\begin{align}\label{3.a5}
\int_{Y'}|\nabla w_{1}|^2\,dx \leq C\left(\int_{Y'}|\nabla w_{1}|^2\right)^{1/2}.
\end{align}
This implies \eqref{3.24} hold.

For case $i=2$. Instead of \eqref{3.101}, we can use
\begin{align}\label{3.109}
\mathcal{T}_{22}^{11}=4\kappa_{0}^2\frac{x_{1}^2x_{2}^2}{\delta^3(x_{1})},\quad
\mathcal{T}_{21}^{12}=\frac{-4\kappa_{0}^2(\lambda+\mu)}{\mu}\frac{x_{1}^2x_{2}^2}{\delta^3(x_{1})},
\end{align}
such that
$$\partial_{x_{2}}\mathcal{T}_{22}^{11}=\mathcal{R}_{22}^{11},
\quad\quad\partial_{x_{2}}\mathcal{T}_{21}^{12}=\mathcal{R}_{21}^{12},$$
to obtain
\eqref{3.24},
by the same way as in case  $i=1$.
\end{proof}
For $|z_{1}|\leq r/4,~s<r/4$, set
$$\Omega_{s}(z_{1}):=\{(x_{1},x_{2})\,|-\frac{\epsilon}{2}+h_{2}(x_{1})<x_{2}<\frac{\epsilon}{2}+h_{1}(x_{1}),
\,|x_{1}-z_{1}|<s\}.$$
Now we use  iteration technique developed in \cite{BLL1} to estimates the scale of the local energy of $w_{i}$ in a small square region $\Omega_{\delta}(z_{1})$.
\begin{lemma}\label{lem3.2}
\begin{align}\label{3.70}
\int_{\Omega_{\delta}(z_{1})}|\nabla w_{i}|^2\,dx\leq C\delta^{2}(z_{1}).
\end{align}
\end{lemma}
\begin{proof}
The following iteration scheme we used is similar in spirit to that  in \cite{BLL1}. For $0<t<s<r/2$,
let $\eta$ be a smooth function satisfying $\eta(x_{1})=1$ if $|x_{1}-z_{1}|<t,\,\eta(x_{1})=0$ if
$|x_{1}-z_{1}|>s,\,0\leq \eta(x_{1})\leq1$ if $t\leq|x_{1}-z_{1}|\leq s$, and $|\eta'(x_{1})|\leq\frac{2}{s-t}$.
Multiplying the equation in \eqref{3.25} by $w\eta^{2}$ and integrating by parts leads to
the following inequality, the same as in (3.30) in \cite{BLL1},
\begin{align*}
\int_{\Omega_{t}(z_{1})}|\nabla w_{i}|^2\,dx\leq\frac{C}{(s-t)^2}\int_{\Omega_{s}(z_{1})}|w_{i}|^2\,dx
+(s-t)^2\int_{\Omega_{s}(z_{1})}|\mathcal{L}_{\lambda,\mu}u_{i}|^2\,dx.
\end{align*}
Note that for $0<s<\frac{2 |z_{1}|}{3}$, see (3.31) in \cite{BLL1},
\begin{align*}
\int_{\Omega_{s}(z_{1})}|w_{i}|^2\,dx\leq C\delta^{2}(z_{1})\int_{\Omega_{s}(z_{1})}|\nabla w_{i}|^2\,dx.
\end{align*}
By \eqref{Lu_1} and \eqref{Lu_2}, we  have
\begin{align*}
\int_{\Omega_{s}(z_{1})}|\mathcal{L}_{\lambda,\mu}u_{i}|^{2}\,dx\leq\frac{Cs}{\delta(z_{1})},\quad 0<s<\frac{2|z_{1}|}{3},
\end{align*}
which exactly is an improvement of (3.32) and (3.35) in \cite{BLL1}.
Denote
\begin{align*}
F(t):=\int_{\Omega_{t}(z_{1})}|\nabla w_{i}|^2\,dx.
\end{align*}
It follows from the above that
\begin{align}\label{3.63}
F(t)\leq\left(\frac{C_{0}\delta(z_{1})}{s-t}\right)^{2}F(s)+C(s-t)^2\frac{s}{\delta(z_{1})},\quad\forall\, 0<t<s<\frac{2|z_{1}|}{3},
\end{align}
where $C_{0}$ is also a universal constant.

Let $t_{j}=2C_{0}j\delta(z_{1}),\,j=1,2,\dots,$ then
\begin{align*}
\frac{C_{0}\delta(z_{1})}{t_{j+1}-t_{j}}=\frac{1}{2},
\end{align*}
taking $s=t_{j+1}$ and $t=t_{j}$ in \eqref{3.63}, we have
\begin{align*}
F(t_{j})\leq\frac{1}{4}F(t_{j+1})+\frac{C(t_{j+1}-t_{j})^2t_{j+1}}{\delta(z_{1})}
\leq\frac{1}{4}\widehat{F}(t_{j+1})+C(j+1)\delta^{2}(z_{1}).
\end{align*}
After $k=\left[\frac{1}{4C_{0}\sqrt{\delta(z_{1})}}\right]$ iterations, and using \eqref{3.24}, we have
\begin{align*}
F(t_{1})&\leq\left(\frac{1}{4}\right)^{k}F(t_{k+1})+C\delta^{2}(z_{1})
\sum\limits_{l=1}^{k}\left(\frac{1}{4}\right)^{l-1}(l+1)\\
&\leq C\left(\frac{1}{4}\right)^{k}+C\delta^{2}(z_{1})\sum\limits_{l=1}^{k}\left(\frac{1}{4}\right)^{l-1}(l+1)\\
&\leq C\delta^{2}(z_{1}).
\end{align*}
This means that
\begin{align*}
\int_{\widehat{\Omega}_{s}(z_{1})}|\nabla w_{i}|^2\,dx\leq C\delta^{2}(z_{1}).
\end{align*}
Then \eqref{3.70} is proved.
\end{proof}
\begin{proof}[Proof of Proposition \ref{prop1}]
By the scaling argument, $W^{2,p}$ estimate bootstrap argument, and embedding theorem, it follows from (3.40) in\cite{BLL1} that
\begin{align}\label{3.69}
\left\|\nabla w_{i}\right\|_{L^{\infty}(\Omega_{\frac{\delta}{2}}(z_{1}))}\leq\frac{C}{\delta}
\left(\|\nabla w_{i}\|_{L^{2}(\Omega_{\delta}(z_{1}))}
+\delta^{2}\|\mathcal{L}_{\lambda,\mu}u_{i}\|_{L^{\infty}(\Omega_{\delta}(z_{1}))}\right).
\end{align}
Using \eqref{3.70},
\begin{align*}
\int_{\Omega_{\delta}(z_{1})}|\nabla w_{i}|^2\,dx\leq C\delta^{2}(z_{1}).
\end{align*}
By \eqref{Lu_1} and \eqref{Lu_2}, we have
\begin{align*}
\delta^{2}\left|\mathcal{L}_{\lambda,\mu}u_{i}\right|\leq\delta^{2}\frac{C}{\delta(z_{1})}
\leq C\delta(z_{1}),\quad \hbox{in}\,\,\Omega_{\delta}(z_{1}).
\end{align*}
Thus, we deduce from \eqref{3.69} that
\begin{align*}
|\nabla w_{i}(z_{1},x_{2})|\leq\frac{C\delta(z_{1})}{\delta}\leq C,\quad
\forall\, -\frac{\epsilon}{2}+h_{2}(z_{1})<x_{2}<\frac{\epsilon}{2}+h_{1}(z_{1}).
\end{align*}
The proof of \eqref{3.68} is completed.
\end{proof}

\subsection{Proof of Theorem \ref{thm3.1}}
Notice that the components $C_{ijkl}$ possess symmetry property:
\begin{align*}
C_{ijkl}=C_{klij}=C_{klji},\quad i,j,k,l=1,2.
\end{align*}
For $2 \times 2$ matrices $A=(A_{ij}),~B=(B_{ij})$, denote
$$(\mathbb{C}A)_{ij}=\sum_{k,l=1}^{2}C_{ijkl}A_{kl},\quad\mbox{and}\quad
(A,B)\equiv A:B=\sum_{i,j=1}^{2}A_{ij}B_{ij}.$$
Clearly,
\begin{align}\label{1.5}
(\mathbb{C}A,B)=(A,\mathbb{C}B).
\end{align}
Therefore, for $i=1,2$  we have
\begin{align}\label{3.w}
(\mathbb{C}e({v}_{i}),e({v}_{i}))=(\mathbb{C}\nabla{v}_{i},\nabla{v}_{i}).
\end{align}
\begin{proof}[Proof of Theorem \ref{thm3.1}]

Recalling $w_{i}:=v_{i}-u_{i}$, and in view of \eqref{1.5} and \eqref{3.w}, we divide $(\mathbb{C}e({v}_{i}),e({v}_{i}))$ into three parts,
\begin{align}\label{3.31}
(\mathbb{C}e({v}_{i}),e({v}_{i}))
=&\,(\mathbb{C}\nabla u_{i},\nabla u_{i})+2(\mathbb{C}\nabla u_{i},\nabla w_{i})+(\mathbb{C}\nabla w_{i},\nabla w_{i})
\nonumber\\
:=&\,\mathrm{I}+\mathrm{II}+\mathrm{III}.
\end{align}

 For case $i=1$. First, recalling \eqref{3.11} and Proposition \ref{prop1}, we have
\begin{align}\label{3.33}
\Big|\int_{\Omega_{\frac{r}{2}}}\mathrm{II}~dx\Big|
&\leq C\int_{\Omega_{\frac{r}{2}}}|\nabla u_{1}||\nabla w_{1}|\,dx\nonumber\\
&\leq C\int_{\Omega_{\frac{r}{2}}}\left(\frac{1}{\delta(x_{1})}+\frac{|x_{1}|}{\delta(x_{1})}+1\right)\,dx \leq C,
\end{align}
and
\begin{align}\label{3.34}
\Big|\int_{\Omega_{\frac{r}{2}}}\mathrm{III}~dx\Big|
\leq C\int_{\Omega_{\frac{r}{2}}}|\nabla w_{1}|^2dx\leq C.
\end{align}

We further denote
\begin{equation}\label{3.93}
\mathrm{I}=(\mathbb{C}\nabla u_{1},\nabla u_{1})=\mathrm{I}_{11}+\mathrm{I}_{12}+\mathrm{I}_{21}+\mathrm{I}_{22},
\end{equation}
where
\begin{align}
\mathrm{I}_{11}:&\nonumber=\left((\lambda+2\mu)\sum\limits_{i=1}^{2}\partial_{x_{i}}u_{1}^{(1)}+\lambda\sum\limits_{i=1}^{2}\partial_{x_{i}}u_{1}^{(2)} \right)\partial_{x_{1}}u_{1}^{(1)},\\
\mathrm{I}_{12}:&\nonumber=\left(\mu\sum\limits_{i,j=1}^{2}\partial_{x_{i}}u_{1}^{(j)}\right)\partial_{x_{2}}u_{1}^{(1)}
,\qquad\qquad
\mathrm{I}_{21}:=\left(\mu\sum\limits_{i,j=1}^{2}\partial_{x_{i}}u_{1}^{(j)}\right)\partial_{x_{1}}u_{1}^{(2)},\\
\mathrm{I}_{22}:&=\left(\lambda \sum\limits_{i=1}^{2}\partial_{x_{i}}u_{1}^{(1)}+
  (\lambda+2\mu)\sum\limits_{i=1}^{2}\partial_{x_{i}}u_{1}^{(2)}\right)\partial_{x_{2}}u_{1}^{(2)}.\label{3.94}
\end{align}
By observation, we find that among all the terms of $\mathrm{I}$, except three of them,
\begin{align}
\partial_{x_{1}}\bar{u}_{1}^{(1)}\partial_{x_{2}}\bar{u}_{1}^{(1)}=-2\kappa_{0}\frac{x_{1}x_{2}}{\delta^3(x_{1})},&\quad
\left|\partial_{x_{2}}\bar{u}_{1}^{(1)}\right|^{2}=\frac{1}{\delta^2(x_{1})},\nonumber\\
\partial_{x_{2}}\bar{u}_{1}^{(1)}\partial_{x_{2}}\tilde{u}_{1}^{(2)}
=\frac{2\kappa_{0}(\lambda+\mu)}{\lambda+2\mu}&\frac{x_{1}x_{2}}{\delta^3(x_{1})},\label{3.a6}
\end{align}
all the other terms can be controlled by $\frac{C}{\delta(x_{1})}$, by using  \eqref{3.11}. Because
\begin{align}\label{3.a3}
\int_{\Omega_{\frac{r}{2}}}\frac{1}{\delta(x_{1})}\,dx\leq C,
\end{align}
they all are good terms.

Thus, it is easy to see from \eqref{3.94} that
\begin{align}
\left|\int_{\Omega_{\frac{r}{2}}}\mathrm{I}_{21}\,dx\right|&\leq C
\int_{\Omega_{\frac{r}{2}}}\left|\sum\limits_{i=1}^{2}\left(\partial_{x_{i}}\bar{u}_{1}^{(1)}+\partial_{x_{i}}\tilde{u}_{1}^{(1)}
+\partial_{x_{i}}\tilde{u}_{1}^{(2)}\right)\partial_{x_{1}}\tilde{u}_{1}^{(2)}\right|\,dx \nonumber\\
&\leq C\int_{\Omega_{\frac{r}{2}}}\frac{1}{\delta(x_{1})}dx\leq C.\label{3.05}
\end{align}
On the other hand, since $\frac{x_{1}x_{2}}{\delta^3(x_{1})}$ is an odd function of $x_{2}$, it follows that
\begin{align}\label{3.k}
\int_{\Omega_{\frac{r}{2}}}\partial_{x_{1}}\bar{u}_{1}^{(1)}\partial_{x_{2}}\bar{u}_{1}^{(1)}\,dx
=\int_{\Omega_{\frac{r}{2}}}\partial_{x_{2}}\bar{u}_{1}^{(1)}\partial_{x_{2}}\tilde{u}_{1}^{(2)}\,dx
=\int_{\Omega_{\frac{r}{2}}}\frac{x_{1}x_{2}}{\delta^3(x_{1})}\,dx=0.
\end{align}
So we have
\begin{align}\label{3.o}
\int_{\Omega_{\frac{r}{2}}}\mathrm{I}_{11}\,dx=(\lambda+2\mu)\int_{\Omega_{\frac{r}{2}}}\partial_{x_{2}}\bar{u}_{1}^{(1)}\partial_{x_{1}}\bar{u}_{1}^{(1)}\,dx
+O(1)=0+O(1),
\end{align}
and
\begin{align}\label{3.p}
\int_{\Omega_{\frac{r}{2}}}\mathrm{I}_{22}\,dx=\lambda\int_{\Omega_{\frac{r}{2}}}\partial_{x_{2}}\bar{u}_{1}^{(1)}\partial_{x_{2}}\tilde{u}_{1}^{(2)}\,dx
+O(1)=0+O(1).
\end{align}

Now for $\mathrm{I}_{12}$, we write it as
\begin{align}
\mathrm{I}_{12}=&\mu\partial_{x_{1}}u_{1}^{(1)}\partial_{x_{2}}u_{1}^{(1)}+\mu\left|\partial_{x_{2}}u_{1}^{(1)}\right|^{2}
+\mu\partial_{x_{1}}u_{1}^{(2)}\partial_{x_{2}}u_{1}^{(1)}
+\mu\partial_{x_{2}}u_{1}^{(2)}\partial_{x_{2}}u_{1}^{(1)}
\nonumber\\
:=&\mathrm{I}_{12}^{1}+\mathrm{I}_{12}^{2}+\mathrm{I}_{12}^{3}+\mathrm{I}_{12}^{4}.\label{3.r}
\end{align}
By using \eqref{3.a6} and \eqref{3.k}, we have
\begin{align}\label{3.a8}
\int_{\Omega_{\frac{r}{2}}}\mathrm{I}_{12}^{1}\,dx
&=\mu\int_{\Omega_{\frac{r}{2}}}\partial_{x_{1}}\bar{u}_{1}^{(1)}\partial_{x_{2}}\bar{u}_{1}^{(1)}\,dx
+O(1)=0+O(1),
\end{align}
and
\begin{align}\label{3.103}
\int_{\Omega_{\frac{r}{2}}}\mathrm{I}_{12}^{4}\,dx
=\mu\int_{\Omega_{\frac{r}{2}}}\partial_{x_{2}}\tilde{u}_{1}^{(2)}\partial_{x_{2}}\bar{u}_{1}^{(1)}\,dx+O(1)=0+O(1).
\end{align}
By \eqref{3.a3},
\begin{align}\label{3.106}
\left|\int_{\Omega_{\frac{r}{2}}}\mathrm{I}_{12}^{3}dx\right|
=\left|\int_{\Omega_{\frac{r}{2}}}\mu\partial_{x_{1}}\tilde{u}_{1}^{(2)}(\partial_{x_{2}}\bar{u}_{1}^{(1)}
+\partial_{x_{2}}\tilde{u}_{1}^{(1)})dx\right|
\leq C\int_{\Omega_{\frac{r}{2}}}\frac{1}{\delta(x_{1})}dx\leq C.
\end{align}
Recalling that $u_{1}=\bar{u}_{1}+\tilde{u}_{1}$, we have
\begin{align}\label{3.a7}
\int_{\Omega_{\frac{r}{2}}}\mathrm{I}_{12}^{2}\,dx
&=\mu\int_{\Omega_{\frac{r}{2}}}\left|\partial_{x_{2}}\bar{u}_{1}^{(1)}\right|^{2}dx+O(1)\nonumber\\
&=\mu\int_{|x_{1}|<\frac{r}{2}}
\frac{1}{\delta(x_{1})}\,dx_{1}+O(1)
=\frac{\mu\pi}{\sqrt{\kappa_{0}}}\frac{1}{\sqrt{\epsilon}}+O(1).
\end{align}
This, together with \eqref{3.r}-\eqref{3.106}, yields
\begin{align}
\int_{\Omega_{\frac{r}{2}}}\mathrm{I}_{12}\,dx&=\int_{\Omega_{\frac{r}{2}}}
(\mathrm{I}_{12}^{1}+\mathrm{I}_{12}^{2}+\mathrm{I}_{12}^{3}+\mathrm{I}_{12}^{4})\,dx
=\frac{\mu\pi}{\sqrt{\kappa_{0}}}\frac{1}{\sqrt{\epsilon}}+O(1).\label{3.a2}
\end{align}

Thus, combining \eqref{3.05}, \eqref{3.o} and \eqref{3.p}, we obtain
\begin{align}\label{3.32}
\int_{\Omega_{\frac{r}{2}}}(\mathbb{C}\nabla u_{1},\nabla u_{1})dx&=\int_{\Omega_{\frac{r}{2}}}\mathrm{I}_{12}\,dx
+\int_{\Omega_{\frac{r}{2}}}(\mathrm{I}_{11}+\mathrm{I}_{21}+\mathrm{I}_{22})\,dx
=\frac{\mu\pi}{\sqrt{\kappa_{0}}}\frac{1}{\sqrt{\epsilon}}+O(1).
\end{align}
Therefore, instituting \eqref{3.33}, \eqref{3.34} and \eqref{3.32} into \eqref{3.31}, combining with \eqref{eq3.6}, we have
\begin{align}\label{3.a9}
\mathcal{E}_{1}
&=\int_{\Omega_{\frac{r}{2}}}(\mathbb{C}{\nabla}v_{1},{\nabla}v_{1})dx
+\int_{Y'\backslash\Omega_{\frac{r}{2}}}(\mathbb{C}{\nabla}v_{1},{\nabla}v_{1})dx\nonumber\\
&=\int_{\Omega_{\frac{r}{2}}}(\mathrm{I}+\mathrm{II}+\mathrm{III})~dx+O(1)\nonumber\\
&=\frac{\mu\pi}{\sqrt{\kappa_{0}}}\frac{1}{\sqrt{\epsilon}}+O(1).
\end{align}

For case $i=2$. The process is the same. We only point out the differences. By \eqref{3.40} and Proposition \ref{prop1}, we have
\begin{align}\label{3.a4}
\Big|\int_{\Omega_{\frac{r}{2}}}(\mathrm{II}+\mathrm{III})\,dx\Big|\leq \int_{\Omega_{\frac{r}{2}}}\Big|2(\mathbb{C}\nabla u_{2},\nabla w_{2})+(\mathbb{C}\nabla w_{2},\nabla w_{2})\,\Big|dx\leq C.
\end{align}
Let
\begin{align}\label{3.q}
\mathrm{I}:=(\mathbb{C}\nabla u_{2},\nabla u_{2})=\mathrm{I}_{11}+\mathrm{I}_{12}+\mathrm{I}_{21}+\mathrm{I}_{22},
\end{align}
where
\begin{align}\label{3.95}
\mathrm{I}_{11}:&=\left((\lambda+2\mu)\sum\limits_{i=1}^{2}\partial_{x_{i}}u_{2}^{(1)}+\lambda\sum\limits_{i=1}^{2}
\partial_{x_{i}}u_{2}^{(2)} \right)\partial_{x_{1}}u_{2}^{(1)},\nonumber\\
\mathrm{I}_{12}:&=\left(\mu\sum\limits_{i,j=1}^{2}\partial_{x_{i}}u_{2}^{(j)}\right)\partial_{x_{2}}u_{2}^{(1)},\qquad\qquad
\mathrm{I}_{21}:=\left(\mu\sum\limits_{i,j=1}^{2}\partial_{x_{i}}u_{2}^{(j)}\right)\partial_{x_{1}}u_{2}^{(2)},\nonumber\\
\mathrm{I}_{22}:&=\left(\lambda \sum\limits_{i=1}^{2}\partial_{x_{i}}u_{2}^{(1)}+
  (\lambda+2\mu)\sum\limits_{i=1}^{2}\partial_{x_{i}}u_{2}^{(2)}\right)\partial_{x_{2}}u_{2}^{(2)}.
\end{align}
Similarly as before, among all the terms of $\mathrm{I}$, except three of them
\begin{align}\label{3.l}
\partial_{x_{1}}\bar{u}_{2}^{(2)}\partial_{x_{2}}\bar{u}_{2}^{(2)}&=-2\kappa_{0}\frac{x_{1}x_{2}}{\delta^3(x_{1})}
,\quad\left|\partial_{x_{2}}\bar{u}_{2}^{(2)}\right|^{2}=\frac{1}{\delta^2(x_{1})}
,\nonumber\\
\partial_{x_{2}}\bar{u}_{2}^{(2)}\partial_{x_{2}}\tilde{u}_{2}^{(1)}
&=\frac{2\kappa_{0}(\lambda+\mu)}{\mu}\frac{x_{1}x_{2}}{\delta^3(x_{1})},
\end{align}
all the others can be controlled by $\frac{C}{\delta(x_{1})}$, by using \eqref{3.40}.
In the view of \eqref{3.k}, it follows that
\begin{align*}
\int_{\Omega_{\frac{r}{2}}}
\partial_{x_{1}}\bar{u}_{2}^{(2)}\partial_{x_{2}}\bar{u}_{2}^{(2)}\,dx
=\int_{\Omega_{\frac{r}{2}}}\partial_{x_{2}}\bar{u}_{2}^{(2)}\partial_{x_{2}}\tilde{u}_{2}^{(1)}\,dx
=\int_{\Omega_{\frac{r}{2}}}\frac{x_{1}x_{2}}{\delta^3(x_{1})}\,dx=0.
\end{align*}
Then,
\begin{align}
\int_{\Omega_{\frac{r}{2}}}\mathrm{I}\,dx&=\int_{\Omega_{\frac{r}{2}}}\mathrm{I}_{22}\,dx
+O(1)=(\lambda+2\mu)\int_{\Omega_{\frac{r}{2}}}\left|\partial_{x_{2}}\bar{u}_{2}^{(2)}\right|^{2}\,dx+O(1)\nonumber\\
&=\mu\int_{\Omega_{\frac{r}{2}}}\frac{1}{\delta^2(x_{1})}\,dx+O(1)
=\frac{(\lambda+2\mu)\pi}{\sqrt{\kappa_{0}}}\frac{1}{\sqrt{\epsilon}}+O(1).\label{3.46}
\end{align}
By\eqref{eq3.6}, \eqref{3.31}, \eqref{3.a4} and \eqref{3.46} we have
\begin{align}\label{3.a15}
\mathcal{E}_{2}
&=\int_{\Omega_{\frac{r}{2}}}(\mathbb{C}{\nabla}v_{2},{\nabla}v_{2})dx
+\int_{Y'\backslash\Omega_{\frac{r}{2}}}(\mathbb{C}{\nabla}v_{2},{\nabla}v_{2})dx\nonumber\\
&=\int_{\Omega_{\frac{r}{2}}}\mathrm{I}~dx+O(1)=\frac{(\lambda+2\mu)\pi}{\sqrt{\kappa_{0}}}\frac{1}{\sqrt{\epsilon}}+O(1).
\end{align}
The proof of Theorem \ref{thm3.1} is completed.
\end{proof}

\section{Proof of Theorem \ref{thm2.2}}
In this section  we consider the $m$-convex inclusion, which is the curvilinear square with rounded-off angles, namely $|x_{1}|^m+|x_{2}|^m\leq r^m$.  Instead of Theorem \ref{thm3.1}, we have
\begin{figure}
  \centering
   \hspace{2cm}\includegraphics[width=0.5\textwidth]{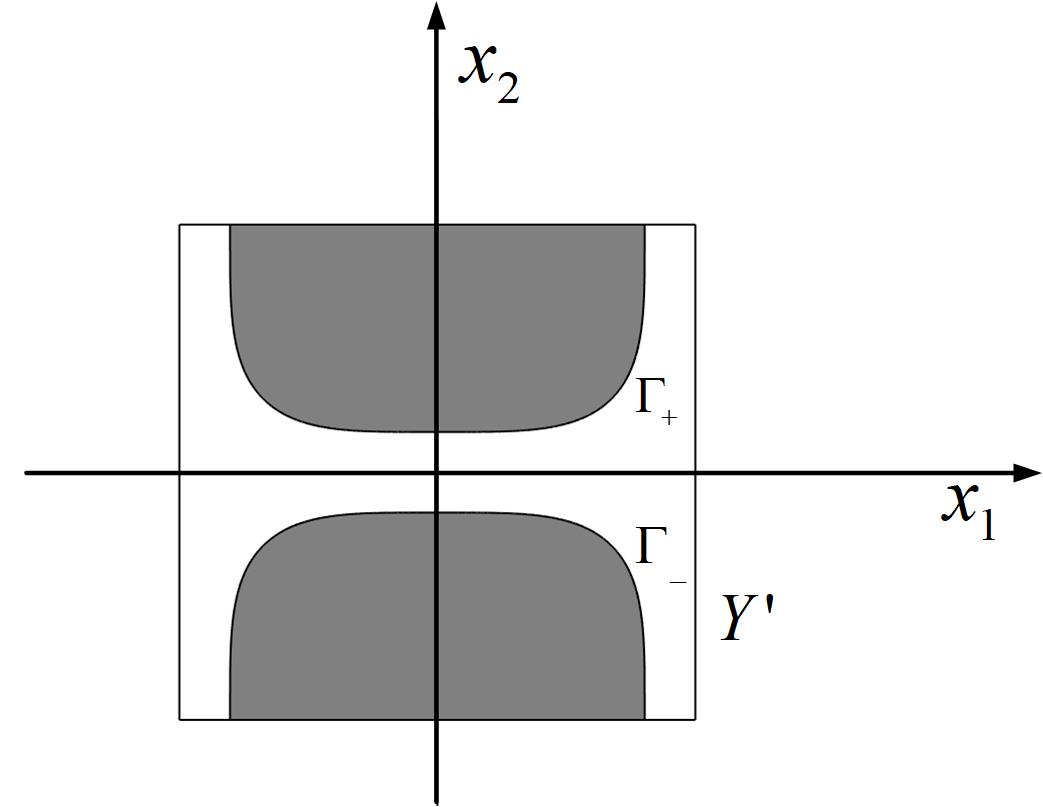}\\
  \caption{$m$-covex inclusions, $m=4$.}
  \label{Fig.5}
\end{figure}
\begin{theorem}\label{thm4.1}
Let $m>2$, then energy integral $\mathcal{E}_{i}$ as in \eqref{2.1}  has the following expansions,
\begin{align*}
\mathcal{E}_{1}&=2\mu
\frac{\pi}{m\sin{\frac{\pi}{m}}}\frac{1}{{\kappa_{0}}^{\frac{1}{m}}}\frac{1}{\epsilon^{1-\frac{1}{m}}}+O(1),
\end{align*}
and
\begin{align*}
\mathcal{E}_{2}&=2(\lambda+2\mu)
\frac{\pi}{m\sin{\frac{\pi}{m}}}\frac{1}{{\kappa_{0}}^{\frac{1}{m}}}\frac{1}{\epsilon^{1-\frac{1}{m}}}+O(1),
\end{align*}
where $\kappa_{0}=2r^{1-m}/m$, as $\epsilon\rightarrow 0$.
\end{theorem}
Thus, Theorem \ref{thm2.2} is an immediate consequence. In the following we give some elementary estimates of $u_{i}$, for $m>2$.

In this case, for $i=1$, the auxiliary function can be constructed in $\Omega_{\frac{r}{2}}$, by a modification of \eqref{u_1}
\begin{align*}
u_{1}=&\bar{u}_{1}+\tilde{u}_{1}:=\frac{2x_{2}+\delta(x_{1})}{2\delta(x_{1})}\begin{pmatrix}
1 \\
0
\end{pmatrix}+\Big(\big(\frac{x_{2}}{\delta(x_{1})}\big)^2-\frac{1}{4}~\Big)\begin{pmatrix}
(2-\frac{\mu}{\lambda+2\mu})\frac{\kappa_{0}}{3}\frac{m(m-1)}{2}x_{2}x_{1}^{m-2} \\\\
(1-\frac{\mu}{\lambda+2\mu})\kappa_{0}\frac{m}{2}x_{1}^{m-1}
\end{pmatrix}
\end{align*}
and still satisfies $u_{1}=\Psi_{1}~ \hbox{on} ~\Gamma_{+}, u_{1}=0~\hbox{on}~ \Gamma_{-} ~\hbox{and}~\frac{\partial u_{1}}{\partial \nu_{0}}\Big|_{+}=0~\hbox{on}~x_{1}=\pm L_{1}$.
For simplify of calculation, we still assume that
\begin{equation*}
h_{1}(x_{1})=\frac{\kappa_{0}}{2}|x_{1}|^{m}\quad\hbox{and}\quad h_{2}(x_{1})=-\frac{\kappa_{0}}{2}|x_{1}|^{m},
\end{equation*}
where $\kappa_{0}=2r^{1-m}/m$. Then
$$ \delta(x_{1})=\epsilon+\kappa_{0}|x_{1}|^{m}.$$
A direct calculation gives
\begin{align}
&\partial_{x_{1}}\bar{u}_{1}^{(1)}=-m\kappa_{0}\frac{x_{1}^{m-1}x_{2}}{\delta^2(x_{1})},~~~
\partial_{x_{2}}\bar{u}_{1}^{(1)}=\frac{1}{\delta(x_{1})},~~~
~~~\partial_{x_{2}}\tilde{u}_{1}^{(2)}
=\frac{ m\kappa_{0}(\lambda+\mu)}{\lambda+2\mu}\frac{x_{1}^{m-1}x_{2}}{\delta^2(x_{1})},\label{4.4}
\end{align}
and
\begin{align}
|\partial_{x_{1}}\bar{u}_{1}^{(1)}|,\,|\partial_{x_{2}}\tilde{u}_{1}^{(2)}|\leq \frac{C|x_{1}|^{m-1}}{\delta(x_{1})},\quad\quad\quad|\partial_{x_{1}}\tilde{u}_{1}^{(1)}|,\,|\partial_{x_{2}}\tilde{u}_{1}^{(1)}|,\,
|\partial_{x_{1}}\tilde{u}_{1}^{(2)}|\leq C.\label{4.7}
\end{align}

Further,
\begin{align}
\partial_{x_{1}x_{1}}\bar{u}_{1}^{(1)}
&=-m\kappa_{0}(m-1)\frac{x_{1}^{m-2}x_{2}}{\delta^2(x_{1})}+\mathcal{R}_{11}^{11},
\quad\partial_{x_{1}x_{2}}\bar{u}_{1}^{(1)}=-m\kappa_{0}\frac{x_{1}^{m-1}}{\delta^2(x_{1})},\label{4.8}\\
\partial_{x_{2}x_{2}}\tilde{u}_{1}^{(1)}
&=\frac{m\kappa_{0}(m-1)(2\lambda+3\mu)}{\lambda+2\mu}\frac{ x_{1}^{m-2}x_{2}}{\delta^2(x_{1})},~~~~
\partial_{x_{2}x_{2}}\tilde{u}_{1}^{(2)}=\frac{m\kappa_{0}(\lambda+\mu)}{\lambda+2\mu}\frac{ x_{1}^{m-1}}{\delta^2(x_{1})}, \label{4.61}\\
\partial_{x_{1}x_{2}}\tilde{u}_{1}^{(2)}&=\frac{m\kappa_{0}(m-1)(\lambda+\mu)}{\lambda+2\mu}
\frac{x_{1}^{m-2}x_{2}}{\delta^2(x_{1})}
+\mathcal{R}_{12}^{12},\label{4.12}
\end{align}
where
\begin{align}\label{4.10}
\mathcal{R}_{11}^{11}=2m^{2}\kappa_{0}^2\frac{x_{1}^{2m-2}x_{2}}{\delta^3(x_{1})},\quad
\mathcal{R}_{12}^{12}=-\frac{2m^{2}\kappa_{0}^2(\lambda+\mu)}{\lambda+2\mu}\frac{x_{1}^{2m-2}x_{2}}{\delta^3(x_{1})}.
\end{align}
From the above, we can see that \eqref{3.b} and \eqref{3.d} still hold. Because
\begin{align}
&|\mathcal{R}_{11}^{11}|,\,|\mathcal{R}_{12}^{12}|\leq \frac{C|x_{1}|^{m-2}}{\delta(x_{1})},
\quad |\partial_{x_{1}x_{1}}\tilde{u}_{1}^{(1)}|\leq C, \label{4.13}\\
&|\partial_{x_{1}x_{2}}\tilde{u}_{1}^{(1)}|,\,|\partial_{x_{1}x_{1}}\tilde{u}_{1}^{(2)}|\leq C|x_{1}|^{m-3},\label{4.14}
\end{align}
estimate \eqref{Lu_1} becomes
\begin{equation}\label{4.15}
\left|\mathcal{L}_{\lambda,\mu}u_{1}\right|\leq
C\Big(\frac{|x_{1}|^{m-2}}{\delta(x_{1})}+1\Big).
\end{equation}

For $i=2$, $u_{2}$  in $\Omega_{\frac{r}{2}}$  is modification of \eqref{u_2}
\begin{align*}
u_{2}=&\bar{u}_{2}+\tilde{u}_{2}:=\frac{2x_{2}+\delta(x_{1})}{2\delta(x_{1})}\begin{pmatrix}
0\\
1
\end{pmatrix}+\Big(\big(\frac{x_{2}}{\delta(x_{1})}\big)^2-\frac{1}{4}~\Big)\begin{pmatrix}
\frac{\lambda+\mu}{\mu}\kappa_{0}\frac{m}{2}x_{1}^{m-1} \\\\
-\frac{\lambda}{3\mu}\kappa_{0}\frac{ m(m-1)}{2}x_{2}x_{1}^{m-2}
\end{pmatrix}
\end{align*}
and still satisfies the same boundary conditions to \eqref{u_2}.
A direct calculation gives
\begin{align}
\partial_{x_{1}}\bar{u}_{2}^{(2)}
=-m\kappa_{0}\frac{x_{1}^{m-1}x_{2}}{\delta^2(x_{1})},~~~\partial_{x_{2}}\bar{u}_{2}^{(2)}&=\frac{1}{\delta(x_{1})},~~~
\partial_{x_{2}}\tilde{u}_{2}^{(1)}
=\frac{m\kappa_{0}(\lambda+\mu)}{\mu}\frac{x_{1}^{m-1}x_{2}}{\delta^2(x_{1})},\label{4.35}
\end{align}
and
\begin{align}
&|\partial_{x_{1}}\bar{u}_{2}^{(2)}|,\,|\partial_{x_{2}}\tilde{u}_{2}^{(1)}|\leq\frac{ C|x_{1}|^{m-1}}{\delta(x_{1})},\quad\quad\quad
|\partial_{x_{1}}\tilde{u}_{2}^{(1)}|,\,|\partial_{x_{1}}\tilde{u}_{2}^{(2)}|,\,
|\partial_{x_{2}}\tilde{u}_{2}^{(2)}|\leq C.\label{4.38}
\end{align}
Further,
\begin{align}
\partial_{x_{1}x_{1}}\bar{u}_{2}^{(2)}
&=-m\kappa_{0}(m-1)\frac{x_{1}^{m-2}x_{2}}{\delta^2(x_{1})}+\mathcal{R}_{22}^{11},
\quad\partial_{x_{1}x_{2}}\bar{u}_{2}^{(2)}=-m\kappa_{0}\frac{x_{1}^{m-1}}{\delta^2(x_{1})},\label{4.42}\\
\partial_{x_{2}x_{2}}\tilde{u}_{2}^{(1)}&=\frac{m\kappa_{0}(\lambda+\mu)}{\mu}\frac{x_{1}^{m-1}}{\delta^2(x_{1})},
\quad~~~\partial_{x_{2}x_{2}}\tilde{u}_{2}^{(2)}
=\frac{-\lambda m\kappa_{0}(m-1)}{\mu}\frac{x_{1}^{m-2}x_{2}}{\delta^2(x_{1})},\label{4.40}\\
\partial_{x_{1}x_{2}}\tilde{u}_{2}^{(1)}
&=\frac{m\kappa_{0}(m-1)(\lambda+\mu)}{\mu}\frac{x_{1}^{m-2}x_{2}}{\delta^2(x_{1})}
+\mathcal{R}_{21}^{12},\label{4.41}
\end{align}
where
\begin{align*}
\mathcal{R}_{22}^{11}=2m^{2}\kappa_{0}^2\frac{x_{1}^{2m-2}x_{2}}{\delta^3(x_{1})},
\quad\quad
\mathcal{R}_{21}^{12}=-\frac{2m^{2}\kappa_{0}^{2}(\lambda+\mu)}{\mu}\frac{x_{1}^{2m-2}x_{2}}{\delta^3(x_{1})}.
\end{align*}
From the above, we can see that \eqref{3.107} and \eqref{3.a10} still hold. Since
\begin{align}
&|\mathcal{R}_{22}^{11}|\,,|\mathcal{R}_{21}^{12}|
\leq\frac{C|x_{1}|^{m-2}}{\delta(x_{1})},\quad
|\partial_{x_{1}x_{1}}\tilde{u}_{2}^{(2)}|\leq C,\label{4.47}\\
&|\partial_{x_{1}x_{1}}\tilde{u}_{2}^{(1)}|,\,|\partial_{x_{1}x_{2}}\tilde{u}_{2}^{(2)}|\leq  C|x_{1}|^{m-3}.\label{4.46}
\end{align}
Instead \eqref{Lu_2}, we have
\begin{equation}\label{4.64}
|\mathcal{L}_{\lambda,\mu}u_{2}|
\leq C\big(\frac{|x_{1}|^{m-2}}{\delta(x_{1})}+1\big).
\end{equation}
Recalling $w_{i}:=v_{i}-u_{i}$, we still have  Proposition \ref{prop1} holds. Let us first show  Lemma \ref{lem3.1} in this case.
\begin{proof}
For case $i=1$. Instead of  \eqref{3.101}, we have
$$
\mathcal{T}_{11}^{11}:=m^{2}\kappa_{0}^2\frac{x_{1}^{2m-2}x_{2}^{2}}{\delta^3(x_{1})},\quad
\mathcal{T}_{12}^{12}:=-\frac{m^{2}\kappa_{0}^{2}(\lambda+\mu)}{\lambda+2\mu}\frac{x_{1}^{2m-2}x_{2}^{2}}{\delta^3(x_{1})}
.$$
By \eqref{4.7} we still have
\begin{align*}
\int_{\Omega_{r_{0}}}|\partial_{x_{1}}\tilde{u}_{1}^{(1)}|\,dx\leq C\quad\mbox{and}\quad
\int_{\Omega_{r_{0}}}|\partial_{x_{1}}\tilde{u}_{1}^{(2)}|\,dx\leq C.
\end{align*}
So, we obtain $\int_{Y'}|\nabla w_{1}|^2dx\leq C$.

For case $i=2$.
Instead of \eqref{3.109}, we have
$$
\mathcal{T}_{22}^{11}=m^{2}\kappa_{0}^2\frac{x_{1}^{2m-2}x_{2}^{2}}{\delta^3(x_{1})},\quad
\mathcal{T}_{21}^{12}=-\frac{m^{2}\kappa_{0}^2(\lambda+\mu)}{\mu}\frac{x_{1}^{2m-2}x_{2}^{2}}{\delta^3(x_{1})}
$$
to obtain $\int_{Y'}|\nabla w_{2}|^2dx\leq C,$ by the same way as in case $i=1$.
\end{proof}
Next, we prove that Lemma \ref{lem3.2} is also true.
\begin{proof}
For $i=1,2$, by \eqref{4.15} and \eqref{4.64}, we have
\begin{align*}
\int_{\widehat{\Omega}_{s}(z_{1})}|\mathcal{L}_{\lambda,\mu}u_{i}|^{2}\,dx&\leq
\int_{\widehat{\Omega}_{s}(z_{1})}\Big(\frac{|x_{1}|^{m-2}}{\delta(x_{1})}+1\Big)^{2}\,dx\\
&\leq Cs\frac{|z_{1}|^{2m-4}}{\delta(z_{1})},\quad 0<s<\frac{2|z_{1}|}{3}.
\end{align*}
With step-length $2C_{0}\delta(z_{1})$, after $k=\left[\frac{1}{4C_{0}(\delta(z_{1}))^{\frac{m-1}{m}}}\right]$ iterations, as in Lemma \ref{lem3.2}, using \eqref{3.24}, we have
\begin{align}\label{4.23}
\int_{\widehat{\Omega}_{s}(z_{1})}|\nabla w_{i}|^2\,dx\leq C\delta^{2}(z_{1})|z_{1}|^{2m-4},\quad i=1,2.
\end{align}
Then Lemma \ref{lem3.2} is proved.
\end{proof}
Finally, we prove that  Proposition \ref{prop1} is true in this case.
\begin{proof}
For $i=1,2$, by \eqref{4.15} and \eqref{4.64}, we have
\begin{align*}
\delta^{2}\left|\mathcal{L}_{\lambda,\mu}u_{i}\right|\leq\delta^{2}\frac{C|z_{1}|^{m-2}}{\delta}
\leq C\delta(x_{1})|z_{1}|^{m-2},\quad \hbox{in}\,\,\Omega_{\delta}(z_{1}),
\end{align*}
combining with \eqref{4.23}, we deduce from \eqref{3.69} that
\begin{align*}
|\nabla w_{i}(x_{2},z_{1})|\leq C|z_{1}|^{m-2}\leq C,\quad
\forall\, -\frac{\epsilon}{2}+h_{2}(z_{1})<x_{2}<\frac{\epsilon}{2}+h_{1}(z_{1}).
\end{align*}
We obtain that $|\nabla w_{i}|\leq C$.
\end{proof}

\begin{proof}[Proof of Theorem \ref{thm4.1}]
For reader's convenience, we only list the key differences. For case $i=1$. By using  \eqref{4.7} and Proposition \ref{prop1}, it follows that
\begin{align*}
\Big|\int_{\Omega_{\frac{r}{2}}}\mathrm{II}~dx\Big|&=\Big|\int_{\Omega_{\frac{r}{2}}}2(\mathbb{C}\nabla u_{1},\nabla w_{1})dx\Big|\leq
\,C\int_{\Omega_{\frac{r}{2}}}\frac{1}{\delta(x_{1})}dx \leq\,C,
\end{align*}
and \eqref{3.34} is still true with no change.

Realling \eqref{3.93} and \eqref{3.94} for the definition of  term $\mathrm{I}$, \eqref{3.a6} becomes
\begin{align}
\partial_{x_{1}}\bar{u}_{1}^{(1)}\partial_{x_{2}}\bar{u}_{1}^{(1)}
&=-m\kappa_{0}\frac{x_{1}^{m-1}x_{2}}{\delta^3(x_{1})},\quad
\left|\partial_{x_{2}}\bar{u}_{1}^{(1)}\right|^{2}=\frac{1}{\delta^2(x_{1})},\nonumber\\
\partial_{x_{2}}\bar{u}_{1}^{(1)}\partial_{x_{2}}\tilde{u}_{1}^{(2)}
&=\frac{m\kappa_{0}(\lambda+\mu)}{(\lambda+2\mu)}\frac{x_{1}^{m-1}x_{2}}{\delta^3(x_{1})},\label{4.62}
\end{align}
expect for these three terms above, all the other terms can be controlled by $\frac{C}{\delta(x_{1})}$, by using \eqref{4.7}. It is known that
\begin{align}\label{4.60}
\int_{\Omega_{\frac{r}{2}}}\frac{1}{\delta(x_{1})}\,dx\leq C.
\end{align}
So, estimate \eqref{3.05} still holds.
Because $\frac{x_{1}^{m-1}x_{2}}{\delta^3(x_{1})}$ is an odd function of $x_{2}$, it follows that
\begin{align}\label{4.63}
\int_{\Omega_{\frac{r}{2}}}\partial_{x_{1}}\bar{u}_{1}^{(1)}\partial_{x_{2}}\bar{u}_{1}^{(1)}\,dx=
\int_{\Omega_{\frac{r}{2}}}\partial_{x_{2}}\bar{u}_{1}^{(1)}\partial_{x_{2}}\tilde{u}_{1}^{(2)}\,dx
=\int_{\Omega_{\frac{r}{2}}}\frac{x_{1}^{m-1}x_{2}}{\delta^3(x_{1})}\,dx=0,
\end{align}
combining with \eqref{4.62}, we obtain that \eqref{3.o} and \eqref{3.p} still hold.

Now divide
$$\mathrm{I}_{12}=\mathrm{I}_{12}^{1}+\mathrm{I}_{12}^{2}+\mathrm{I}_{12}^{3}+\mathrm{I}_{12}^{4},$$
as \eqref{3.r} in proof of Theorem \ref{thm3.1}. First, by \eqref{4.62}, we note that \eqref{3.a8} and \eqref{3.103} still hold. Then in view of \eqref{4.60}, estimate \eqref{3.106} still holds.
\eqref{3.a7} becomes
\begin{align*}
\int_{\Omega_{\frac{r}{2}}}\mathrm{I}_{12}^{2}&=\mu\int_{|x_{1}|<\frac{r}{2}}
\frac{1}{\delta(x_{1})}\,dx_{1}+O(1)
=2\mu
\frac{\pi}{m\sin{\frac{\pi}{m}}}\frac{1}{{\kappa_{0}}^{\frac{1}{m}}}\frac{1}{\epsilon^{1-\frac{1}{m}}}+O(1).
\end{align*}

So,
\begin{align*}
\int_{\Omega_{\frac{r}{2}}}\mathrm{I}\,dx=
\int_{\Omega_{\frac{r}{2}}}\mathrm{I}_{12}dx+O(1)=2\mu
\frac{\pi}{m\sin{\frac{\pi}{m}}}\frac{1}{{\kappa_{0}}^{\frac{1}{m}}}\frac{1}{\epsilon^{1-\frac{1}{m}}}+O(1).
\end{align*}

Therefore, instead of \eqref{3.a9}, we obtain
\begin{align*}
\mathcal{E}_{1}=2\mu
\frac{\pi}{m\sin{\frac{\pi}{m}}}\frac{1}{{\kappa_{0}}^{\frac{1}{m}}}\frac{1}{\epsilon^{1-\frac{1}{m}}}+O(1).
\end{align*}

For case $i=2$. By \eqref{4.38}, we still have \eqref{3.a4} in proof of Theorem \ref{thm3.1}.
Recalling \eqref{3.q} and \eqref{3.95}, we now calculate
$$\mathrm{I}=(\mathbb{C}\nabla u_{2},\nabla u_{2})=\mathrm{I}_{11}+\mathrm{I}_{12}+\mathrm{I}_{21}+\mathrm{I}_{22}.$$
In term $\mathrm{I}$, except these three terms
\begin{align*}
\partial_{x_{1}}\bar{u}_{2}^{(2)}\partial_{x_{2}}\bar{u}_{2}^{(2)}=-m\kappa_{0}\frac{x_{1}^{m-1}x_{2}}{\delta^3(x_{1})},&\quad
\left|\partial_{x_{2}}\bar{u}_{2}^{(2)}\right|^{2}=\frac{1}{\delta^2(x_{1})},\\
\partial_{x_{2}}\bar{u}_{2}^{(2)}\partial_{x_{2}}\tilde{u}_{2}^{(1)}
=\frac{m\kappa_{0}(\lambda+\mu)}{\mu}&\frac{x_{1}^{m-1}x_{2}}{\delta^3(x_{1})},
\end{align*}
by \eqref{4.38} all the other terms can be controlled by $\frac{C}{\delta(x_{1})}$.
In view of \eqref{4.63}, it follows that
\begin{align*}
\int_{\Omega_{\frac{r}{2}}}\partial_{x_{1}}\bar{u}_{2}^{(2)}\partial_{x_{2}}\bar{u}_{2}^{(2)}\,dx
=\int_{\Omega_{\frac{r}{2}}}\partial_{x_{2}}\bar{u}_{2}^{(2)}\partial_{x_{2}}\tilde{u}_{2}^{(1)}\,dx
=\int_{\Omega_{\frac{r}{2}}}\frac{x_{1}^{m-1}x_{2}}{\delta^3(x_{1})}\,dx=0.
\end{align*}
Thus, \eqref{3.a15} becomes
\begin{align*}
\mathcal{E}_{2}
=\int_{\Omega_{\frac{r}{2}}}\mathrm{I}_{22}~dx+O(1)
=\frac{2(\lambda+2\mu)\pi}{m\sin{\frac{\pi}{m}}}\frac{1}{{\kappa_{0}}^{\frac{1}{m}}}\frac{1}{\epsilon^{1-\frac{1}{m}}}+O(1).
\end{align*}
The proof of Theorem \ref{thm4.1} is completed.
\end{proof}

\textbf{Acknowledgements.} Y. Li is grateful to her friends, Dr. LongJuan Xu and ZhiWen Zhao,  for  their helpful suggestions and discussions.

\bibliographystyle{plain}

\def\cprime{$'$}

\end{document}